\documentclass[10pt,a4paper]{article}

\author{%
  Axel B\"ucher\\
  \small Universit\"at Heidelberg \\
  \small Institut f\"ur Angewandte Mathematik\\
  \small Im Neuenheimer Feld 294, 69120 Heidelberg, Germany\\
  \small \texttt{axel.buecher@rub.de}
  \and
  Johan Segers\\
  \small Universit\'e catholique de Louvain\\
  \small Institut de Statistique, Biostatistique et Sciences Actuarielles\\
  \small Voie du Roman Pays 20, B-1348 Louvain-la-Neuve, Belgium\\
  \small \texttt{johan.segers@uclouvain.be}
}

\title{Extreme value copula estimation based on block maxima of a multivariate stationary time series}

\date{\today}

\usepackage[utf8]{inputenc}
\usepackage[authoryear]{natbib}
\usepackage{amsmath}
\usepackage{amsfonts}
\usepackage{amsthm}
\usepackage{amssymb}
\usepackage{bm}
\usepackage{hyperref}
\usepackage{paralist}
\usepackage{graphicx}
\usepackage{booktabs}
\usepackage{enumerate}
\numberwithin{equation}{section}

\usepackage{color}

\newcommand{\floor}[1]{\lfloor{#1}\rfloor}
\newcommand{\ZZ}{\mathbb{Z}}
\newcommand{\RR}{\mathbb{R}}
\newcommand{\NN}{\mathbb{N}}
\newcommand{\reals}{\RR}
\newcommand{\R}{\RR}
\newcommand{\eps}{\varepsilon}
\newcommand{\1}{I}

\newcommand{\id}{{\rm id}}

\newcommand{\diff}{\mathrm{d}}
\newcommand{\ds}{\diff s}
\newcommand{\dy}{\diff y}

\renewcommand{\Finv}{F^\leftarrow}
\newcommand{\Fbinv}{\bm{F}^\leftarrow}

\newcommand{\Ab}{\mathbb{A}}
\newcommand{\Cb}{\mathbb{C}}

\newcommand{\Eb}{\mathbb{E}}

\newcommand{\Wb}{\mathbb{W}}

\newcommand{\Dc}{\mathcal{D}}
\newcommand{\Fc}{\mathcal{F}}

\newcommand{\tl}{{\scriptscriptstyle [\ell]}}
\newcommand{\weak}{{\ \rightsquigarrow\ }}

\renewcommand{\Pr}{\operatorname{P}}
\newcommand{\expec}{\operatorname{E}}
\DeclareMathOperator{\Var}{Var}
\DeclareMathOperator{\Cov}{Cov}
\DeclareMathOperator{\MSE}{MSE}

\newtheorem{theorem}{Theorem}[section]
\newtheorem{lemma}[theorem]{Lemma}
\newtheorem{condition}[theorem]{Condition}

\newtheorem{corollary}[theorem]{Corollary}
\newtheorem{proposition}[theorem]{Proposition}

\theoremstyle{remark}
\newtheorem{example}[theorem]{Example}

\begin{document}

\maketitle

\begin{abstract}
The core of the classical block maxima method consists of fitting an extreme value distribution to a sample of maxima over blocks extracted from an underlying series.
In asymptotic theory, it is usually postulated that the block maxima are an independent random sample of an extreme value distribution.
In practice however, block sizes are finite, so that the extreme value postulate will only hold approximately.
A more accurate asymptotic framework is that of a triangular array of block maxima, the block size depending on the size of the underlying sample in such a way that both the block size and the number of blocks within that sample tend to infinity.
The copula of the vector of componentwise maxima in a block is assumed to converge to a limit, which, under mild conditions, is then necessarily an extreme value copula.
Under this setting and for absolutely regular stationary sequences, the empirical copula of the sample of vectors of block maxima is shown to be a consistent and asymptotically normal estimator for the limiting extreme value copula.
Moreover, the empirical copula serves as a basis for rank-based, nonparametric estimation of the Pickands dependence function of the extreme value copula.
The results are illustrated by theoretical examples and a Monte Carlo simulation study.

\medskip

\end{abstract}

\newpage

\section{Introduction}

The block maximum method for extreme value analysis essentially consists of the following procedure: partition a long series of data into blocks; for each block, compute the maximum; fit an extreme value distribution to the sample of block maxima. Often, the blocks correspond to months or years of data, whence the name `annual maxima series'. The fitted distribution can then be used to compute tail quantiles or `$T$-year return levels'. The approach was developed and popularized in the classic monograph of \cite{gumbel:1958}. The method is applicable even when the individual `daily' observations are unavailable or when the time series exhibits seasonality, as long as the block size is a multiple of the period length. The procedure can be extended to multivariate series too: compute or just observe block maxima for each of variables separately, and fit a multivariate extreme value distribution to the sample of vectors of componentwise block maxima.

The method is justified by the extremal types theorem: under broad conditions, the only possible limits of affinely normalized block maxima, as the block length tends to infinity, are the extreme value distributions. The conditions allow for temporal dependence, provided certain mixing conditions hold; see \cite{leadbetter:etal:1983} for the univariate case and \cite{Hsi89} and \cite{Hus90} for the multivariate case.

Unlike their univariate counterparts, multivariate extreme value distributions do not constitute a parametric family. In statistical applications, a parametric form is often assumed, an early example being \cite{gumbel:mustafi:1967}. In general, the dependence structure or copula should be max-stable. Several representations of max-stable or extreme value copulas exist; see \citet[Chapter~8]{bgst:2004} for an overview. The representation proposed in \cite{pickands:1981} is a popular one and has led to the concept of a Pickands dependence function.

In the large-sample theory for the block maximum method, the data generating process is nearly always specified as independent random sampling from the limiting extreme value distribution. Seminal papers to this view are \cite{prescott:walden:1980} for the univariate case and \cite{tawn:1988}, \cite{tawn:1990} and \cite{deheuvels:1991} for the multivariate case. However, in the light of the above description, this set-up does not correspond to reality for at least two reasons: first, the block maxima are only approximately extreme value distributed, and second, they are only approximately independent.

A first contribution to the mathematical validation of the block maximum method in a more realistic setting is \cite{dombry:2013}. The starting point is a single series of independent and identically distributed univariate random variables whose  distribution is in the domain of attraction of an extreme value distribution. Consistency is shown for the maximum likelihood estimator for the extreme value index applied to the sample of block maxima extracted from the full sample. The block size tends to infinity so that the extremal types theorem can come into force; at the same time, the block size is of smaller order than the sample size, so that the number of blocks, which determines the size of the sample of block maxima, tends to infinity. In the same set-up, the asymptotic distribution of the probability-weighted moment estimator was addressed by Laurens de Haan at the \emph{8th Conference on Extreme Value Analysis} (Fudan University, Shanghai, July 8--12, 2013),
see also the recent working paper \cite{FerHaa13}.

For multivariate time series, nothing has been done in this direction yet, up to the best of our knowledge. The present paper tries to fill this gap. We focus on the estimation of the limit copula of the vector of componentwise block maxima when the block size tends to infinity. The data generating process is a stationary, multivariate time series. Under weak dependence conditions, the limit copula must then be an extreme value copula \citep{Hsi89}. No parametric assumptions are made regarding this extreme value copula. It can be estimated by the empirical copula of the vectors of block maxima. Moreover, the empirical copula can be used as a basis for the nonparametric estimation of the Pickands dependence function of the extreme value copula. For simplicity, we focus on the minimum distance estimator of \cite{BucDetVol11} and \cite{BerBucDet13}, although alternative procedures could have been considered as well \citep{GdS11,peng:qian:yang:2013}.

We study the sequence of empirical copula processes constructed from the triangular array of vectors of block maxima as the block size and the number of blocks tend to infinity. We find that if the underlying series is absolutely regular, the limit process is the same Gaussian process as if the block maxima were sampled independently from a distribution whose copula is already equal to the limiting extreme value copula. This result carries over to the estimation of the Pickands dependence function, where we find the same limit process as in \cite{BerBucDet13}. This does not mean that the temporal dependence can be neglected, however: because of serial dependence, the limiting extreme value copula is in general different from the extreme value attractor of the copula of the stationary distribution of the series. The results are illustrated by means of Monte Carlo simulations.

The structure of the paper is as follows. The objects of interest are described mathematically in Section~\ref{sec:prelim}. The main results on the convergence of the block maxima empirical copula process and the minimum distance estimator for the Pickands dependence function form the subject of Section~\ref{sec:main}. Section~\ref{sec:examples} then contains a number of theoretical examples, whereas Section~\ref{sec:num} reports on the result of a simulation study. Section~\ref{sec:conclusion} concludes. All proofs are collected in the Appendices~\ref{sec:proofs:main} and~\ref{sec:proofs:examples}.

\section{Preliminaries, notations, and assumptions}
\label{sec:prelim}

Consider a $d$-variate stationary time series $\bm{X}_t = (X_{t,1}, \ldots, X_{t,d})$, $t \in \ZZ$. For simplicity, assume that the univariate stationary margins are continuous. A sample of size $n$ is divided into $k$ blocks of length $m$, so that $k = \floor{n/m}$, the integer part of $n/m$, and possibly a remainder block of length $n - km$ at the end. The maximum of the $i$th block in the $j$th component is denoted by
\begin{align*}
  M_{m,i,j} = \max \{ X_{t,j} : t \in (im-m, im] \cap \ZZ \}.
\end{align*}
Let $\bm{M}_{m,i} = (M_{m,i,1}, \ldots, M_{m,i,d})$ be the vector of maxima over the $d$ variables in the $i$th block. For fixed block length $m$, the sequence of block maxima $(\bm{M}_{m,i})_i$ is a stationary process too.

The distributions functions of the block maxima are denoted by
\begin{align*}
  F_m( \bm{x} ) &= \Pr[ \bm M_{m,1} \le \bm{x} ], &
  F_{m,j}( x_j ) &= \Pr[ M_{m,1,j} \le x_j ],
\end{align*}
for $\bm{x} \in \reals^d$ and $j \in \{1, \ldots, d\}$. Observe that $F_1$ is the distribution function of $\bm{X}_1$. If the random vectors $\bm{X}_t$ are serially independent, we have $F_m = F_1^m$. In the general, stationary case, the relation between $F_m$ and $F_1$ is more complex.

The margins of $\bm{X}_1$ being continuous, the margins of $\bm{M}_{m,1}$ are continuous as well. Let $C_m$ be the (unique) copula of $F_m$, which, in the serially independent case, can be written as $C_m(\bm u) = \{C_1(u_1^{1/m}, \dots, u_m^{1/m})\}^m$, $\bm u=(u_1, \dots, u_d) \in [0,1]^d$. In the present context, the domain-of-attraction condition reads as follows.

\begin{condition} 
\label{cond:Cinfty}
There exists a copula $C_\infty$ such that
\[
  \lim_{m \to \infty} C_m(\bm{u}) = C_\infty(\bm{u}) \qquad (\bm{u} \in [0, 1]^d).
\]
\end{condition}

Typically, the limit $C_\infty$ will be an extreme value copula \citep{Hsi89,Hus90}. Below we will assume that the time series $(\bm{X}_t)_t$ is absolutely regular or $\beta$-mixing, which, by Theorem~4.2 in \cite{Hsi89}, is already sufficient for the latter statement. However, $C_\infty$ will in general be different from the extreme value attractor of $C_1$; see for instance Section~\ref{subsec:MM}. If the copula $C_\infty$ in Condition~\ref{cond:Cinfty} is an extreme value copula, it admits the representation
\begin{multline}
\label{eq:Pickands}
  C_\infty(\bm u) 
  = 
  \exp\left\{ \left( \sum_{j=1}^d \log u_j \right) 
    A_\infty \left( \frac{\log u_2}{\sum_{j=1}^d \log u_j}, \dots, \frac{\log u_d}{\sum_{j=1}^d \log u_j} \right) \right\}
\end{multline}
for $\bm u \in [0,1]^d$. Here $A_\infty: \Delta_{d-1} \to [0,1]$ is called the Pickands dependence function of $C_\infty$. It is a convex function defined on the unit simplex $\Delta_{d-1} = \{ \bm{t} = (t_1, \dots, t_{d-1})  \in [0,1]^{d-1} : t_1 + \dots + t_{d-1} \le 1 \}$ and satisfying the bounds $\max \{1 - t_1 - \dots - t_{d-1}, t_1, \dots, t_{d-1} \} \le A_\infty(\bm{t}) \le 1$; see, e.g., \cite{GudSeg10}.

Applying the probability integral transform to the block maxima yields
\begin{align}
\label{eq:Umi}
  U_{m,i,j} &= F_{m,j}( M_{m,i,j} ), &
  \bm{U}_{m,i} &= (U_{m,i,1}, \ldots, U_{m,i,d}).
\end{align}
The random variables $U_{m,i,j}$ are uniformly distributed on $(0, 1)$ and the distribution function of the random vector $\bm{U}_{m,i}$ is the copula $C_m$. The empirical distribution function of the (unobservable) sample $\bm{U}_{m,1}, \ldots, \bm{U}_{m,k}$ is
\begin{equation}
\label{eq:Ccirc}
  \hat{C}_{n,m}^\circ ( \bm{u} )
  =
  \frac{1}{k} \sum_{i=1}^k \1 ( \bm{U}_{m,i} \le \bm{u} ),
\end{equation}
where $\1(A)$ denotes the indicator variable of the event $A$.

Since the marginal distributions $F_{m,j}$ are unknown, we replace them in \eqref{eq:Umi} by their empirical versions $\hat{F}_{n,m,j}$: for $\bm{x} = (x_1, \ldots, x_d) \in \reals^d$,
\begin{align}
\label{eq:empCDF}
  \hat{F}_{n,m}( \bm{x} )
  &= \frac{1}{k} \sum_{i=1}^k \1 ( \bm{M}_{m,i} \le \bm{x} ), &
  \hat{F}_{n,m,j}( x_j )
  &= \frac{1}{k} \sum_{i=1}^k \1 ( M_{m,i,j} \le x_j ).
\end{align}
The resulting `pseudo-observations' are
\begin{align*}
  \hat{U}_{n,m,i,j} &= \hat{F}_{n,m,j}( M_{m,i,j} ), &
  \hat{\bm{U}}_{n,m,i} &= (\hat{U}_{n,m,i,1}, \ldots, \hat{U}_{n,m,i,d}).
\end{align*}
In analogy to \eqref{eq:Ccirc}, the empirical copula is then defined as
\begin{equation}
\label{eq:empcop}
  \hat{C}_{n,m} ( \bm{u} )
  =
  \frac{1}{k} \sum_{i=1}^k \1 ( \hat{\bm{U}}_{n,m,i} \le \bm{u} ).
\end{equation}
In practice, it is customary to divide by $k+1$ rather than by $k$ in \eqref{eq:empCDF}; asymptotically, this does not make a difference. An alternative definition of the empirical copula is via
\begin{equation}
\label{eq:empcopalt}
  \hat{C}_{n,m}^{\mathrm{alt}}( \bm{u} )
  = 
  \hat{F}_{n,m} 
  \bigl( 
    \hat{F}_{n,m,1}^\leftarrow(u_1), \ldots, \hat{F}_{n,m,d}^\leftarrow(u_d) 
  \bigr)
\end{equation}
where $H^\leftarrow$ denotes the left-continuous generalized inverse function of a distribution function $H$, defined as 
\[
  H^{\leftarrow}(p) 
  = 
  \begin{cases}
    \inf \{ x \in \reals : H(x) \ge p \} & \text{if $p \in (0,1]$,} \\
    \sup \{ x \in \reals : H(x) = 0 \} & \text{if $p = 0$.} 
  \end{cases}
\]
In the independent case, it is not difficult to see that the difference between $\hat{C}_{n,m}$ and $\hat{C}_{n,m}^{\mathrm{alt}}$ is bounded in absolute value by $d / k$ almost surely. This difference is asymptotically negligible in view of the $O_p(1/\sqrt{k})$ rate of convergence of $\hat{C}_{n,m}$ that will be established in Theorem~\ref{thm:empcopproc}. However, in the case of serial dependence, the situation is more complicated, because with positive probability, there may be ties among the block maxima, even if their distribution is continuous; see for instance the random-repetition process in Subsection~\ref{subsec:RR}. Nevertheless, we will show in Proposition~\ref{prop:empcopalt} that the difference between $\hat{C}_{n,m}$ and $\hat{C}_{n,m}^{\mathrm{alt}}$ is still $o_p(1/\sqrt{k})$.

The serial dependence in the series $(\bm{X}_t)_t$ is controlled via mixing coefficients. For two $\sigma$-fields $\Fc_1$ and $\Fc_2$ of a probability space $(\Omega, \Fc, \Pr)$, let
\begin{align*}
  \alpha(\Fc_1,\Fc_2) &= \sup_{A \in \Fc_1, B \in \Fc_2} | \Pr(A \cap B) - \Pr(A) \Pr(B) |, \\
  \beta(\Fc_1, \Fc_2) &= \sup \frac{1}{2} \sum_{i,j \in I \times J}  | \Pr(A_i \cap B_j) - \Pr(A_i) \Pr(B_j) |,
\end{align*}
where the latter supremum is taken over all finite partitions $(A_i)_{i\in I}$ and $(B_j)_{j\in J}$ of $\Omega$ consisting of events that are $\Fc_1$ and $\Fc_2$ measurable, respectively. The $\alpha$- and $\beta$-mixing coefficients of a time series $(\bm X_t)_{t\in\ZZ}$, not necessarily stationary, are defined, for $n\ge1$, as
\begin{align}
\label{eq:alpha:beta}
  \alpha(n) =  \sup_{t\in \ZZ} \alpha(\Fc_{-\infty}^{t}, \Fc_{t+n}^\infty),  & & 
  \beta(n) =  \sup_{t\in \ZZ} \beta(\Fc_{-\infty}^{t}, \Fc_{t+n}^\infty),
\end{align}
where, for $-\infty \le \ell_1 < \ell_2 \le \infty$, $\Fc_{\ell_1}^{\ell_2}$ denotes the sigma-field generated by those $\bm X_{t}$ with $t \in [\ell_1, \ell_2] \cap \ZZ$.

Recall that $m$ is the block size and $k = \floor{n/m}$ is the number of blocks. In an asymptotic framework, we consider a block size sequence $m_n$ and the associated block number sequence $k_n = \floor{n/m_n}$.

\begin{condition}
\label{cond:alpha-beta}
There exists a positive integer sequence $\ell_n$ such that the following statements hold:
\begin{compactenum}[(i)]
\item $m_n \to \infty$ and $m_n= o(n)$;
\item $\ell_n \to \infty$ and $\ell_n = o(m_n)$;
\item $k_n \, \alpha(\ell_n) = o(1)$ and $(m_n/\ell_n) \, \alpha(\ell_n) = o(1)$;
\item $\sqrt{k_n} \, \beta(m_n)=o(1)$.
\end{compactenum}
\end{condition}

A sufficient condition for \emph{(iii)--(iv)} is that $(k_n + m_n/\ell_n) \, \beta(\ell_n) = o(1)$.
We will occasionally simplify notation by writing $m = m_n$, $k = k_n$ and $\ell = \ell_n$.

\section{Main results}
\label{sec:main}

The central result of the paper is Theorem~\ref{thm:empcopproc} in Section~\ref{subsec:empcop}, claiming weak convergence of the empirical copula process
\begin{equation}
\label{eq:empcopproc}
  \Cb_{n,m} = \sqrt{k} ( \hat{C}_{n,m} - C_m ).
\end{equation}
To arrive at this result, the case of known margins needs to be treated first; this is done in Section~\ref{subsec:emp}. Weak convergence of $\Cb_{n,m}$ is applied in Section~\ref{subsec:minidi} to find a functional central limit theorem for a rank-based, nonparametric estimator of the Pickands dependence function of the limit copula $C_\infty$.

\subsection{Block maxima empirical process}
\label{subsec:emp}

Weak convergence of the empirical copula process $\Cb_{n,m}$ will follow from the functional delta method provided we have a weak convergence result for the process 
\[ 
  \Cb_{n,m}^{\circ} = \sqrt k ( \hat C_{n,m}^\circ - C_m ), 
\]
where $\hat{C}_{n,m}^\circ$ is defined in \eqref{eq:Ccirc}. If the random variables $\bm U_{m,i}$ were serially independent, then the weak convergence of $\Cb_{n,m}^\circ$ would easily follow from Theorem~2.11.9 in \cite{VanWel96}. The case of serial dependence is reduced to the independence case by a blocking technique and a coupling argument.

\begin{theorem}[Block maxima empirical process]
\label{thm:empproc}
Let $(\bm{X}_t)_{t\in \ZZ}$ be a stationary multivariate time series with continuous univariate margins. If Conditions~\ref{cond:Cinfty} and~\ref{cond:alpha-beta} hold, then
\[
  \Cb_{n,m}^\circ \weak \Cb^\circ \qquad \text{in } \ell^\infty([0,1]^d),
\]
where $\Cb^\circ$ denotes a centered Gaussian process on $[0,1]^d$ with continuous sample paths and covariance structure 
\[
  \expec[\Cb^\circ(\bm u) \Cb^\circ(\bm v)] = C_\infty(\bm u \wedge \bm v) - C_\infty(\bm u) \, C_\infty(\bm v).
\]
\end{theorem}

Interestingly, the limiting process $\Cb^\circ$ is a $C_\infty$-Brownian bridge: the serial dependence between the block maxima has disappeared. The proof of Theorem~\ref{thm:empproc} is given in Appendix~\ref{subsec:proofs:emp}.

\subsection{Block maxima empirical copula process}
\label{subsec:empcop}

Recall the two versions of the empirical copula, $\hat{C}_{n,m}$ in \eqref{eq:empcop} and $\hat{C}_{n,m}^{\mathrm{alt}}$ in \eqref{eq:empcopalt}. By the following proposition, the difference between the two versions is asymptotically negligible. The proofs of Proposition~\ref{prop:empcopalt} and the other results in this section are given in Appendix~\ref{subsec:proofs:empcop}.

\begin{proposition}
\label{prop:empcopalt}
Under the conditions of Theorem~\ref{thm:empproc}, we have
\[
  \sup_{ \bm{u} \in [0, 1]^d }
  \bigl|
    \hat{C}_{n,m}^{\mathrm{alt}}( \bm{u} )
    -
    \hat{C}_{n,m}( \bm{u} )
  \bigr|
  = o_p(1 / \sqrt{k}).
\]
\end{proposition}

It follows that in the definition of the empirical copula process in \eqref{eq:empcopproc}, we can replace $\hat{C}_{n,m}$ by $\hat{C}_{n,m}^{\mathrm{alt}}$, yielding
\[
  \Cb_{n,m}^{\mathrm{alt}} = \sqrt{k} ( \hat{C}_{n,m}^{\mathrm{alt}} - C_m )
\]
at the cost of an $o_p(1)$ term:
\[
  \sup_{ \bm{u} \in [0, 1]^d }
  \bigl|
    \Cb_{n,m}^{\mathrm{alt}}( \bm{u} )
    -
    \Cb_{n,m}( \bm{u} )
  \bigr|
  = o_p(1).  
\]

Now, let us transfer the weak convergence result on $\Cb_{n,m}^\circ$ to $\Cb_{n,m}$. Let $\Dc_\Phi$ denote the set of all cdfs on $[0,1]^d$ whose marginals put no mass at zero. Defining
\begin{equation}
\label{eq:copulaMapping}
  \Phi: \Dc_\Phi \to \ell^\infty([0,1]^d): H \mapsto H(H_1^{\leftarrow}, \dots, H_d^{\leftarrow})
\end{equation}
as the copula mapping, we can write
\[
  \Cb_{n,m}^{\mathrm{alt}}
  = \sqrt{k} \{ \Phi(\hat C_{n,m}^\circ) - \Phi(C_m) \}.
\]
Weak convergence of $\Cb_{n,m}$ and $\Cb_{n,m}^{\mathrm{alt}}$ can be shown by the functional delta method \citep[Section 3.9]{VanWel96}, provided certain smoothness assumptions on the copulas $C_m$ and $C_\infty$ are made, to be introduced next. 

For the limit process to have continuous trajectories, the following condition \citep{segers:2012} is unavoidable and will be assumed throughout.

\begin{condition}
\label{cond:SC}
For any $j=1, \dots, d$, the $j$th first order partial derivative $\dot C_{\infty,j}= \partial C_\infty/\partial u_j $ exists and is continuous on $\{ \bm{u} \in [0,1]^d : u_j \in (0,1) \}$.
\end{condition}

As mentioned right after Condition~\ref{cond:Cinfty}, the $\beta$-mixing condition on the underlying time series in Condition~\ref{cond:alpha-beta} implies that $C_\infty$ is an extreme value copula \citep{Hsi89}. For such copulas, Condition~\ref{cond:SC} has been worked out in Example~5.3 of \cite{segers:2012}, see in particular formula (5.1) therein. In the bivariate case, it is sufficient to assume that the Pickands dependence function is continuously differentiable on $(0,1)$. This is the case for many of the common families of extreme value copulas, as, e.g., the Gumbel, Galambos or H\"usler--Rei\ss\ family.

In addition to Condition~\ref{cond:SC}, some qualification of the convergence of $C_m$ to $C_\infty$ will be needed. We will impose either (a) or (b) of the following condition. Roughly speaking, (a) says that this convergence is sufficiently fast, every subsequence of $\sqrt{k} (C_m - C_\infty)$ containing a further subsequence that converges uniformly, whereas (b) requires locally uniform convergence of the partial derivatives. For $C_m$, these partial derivatives are not supposed to exist, however; instead, we will work with the functions
\begin{equation*}
  \dot{C}_{m,j}( \bm{v} )
  =
  \limsup_{h \searrow 0} h^{-1} \{ C_m( \bm{v} + h \bm{e}_j ) - C_m( \bm{v} ) \},
\end{equation*}
with $\bm{e}_j$ the $j$th canonical unit vector in $\reals^d$, functions which are always defined and which satisfy $0 \le \dot{C}_{m,j} \le 1$ as a consequence of monotonicity and  Lipschitz-continuity of $C_{m}$, its margins being standard uniform. Let $\mathcal{C}([0, 1]^d)$ denote the space of all real-valued, continuous functions on $[0, 1]^d$.

\begin{condition}
\label{cond:C12}
\mbox{}
\begin{compactenum}[(a)]
\item
The sequence $\sqrt k (C_m - C_\infty)$ is relatively compact in $\mathcal{C}([0, 1]^d)$.
\item
For every $\delta \in (0, 1/2)$,
\[
  \max_{j=1,\ldots,d} \sup_{\substack{\bm{u} \in [0, 1]^d: \\ u_j \in [\delta, 1-\delta]}}
  \bigl| \dot{C}_{m,j}( \bm{u} ) - \dot{C}_{\infty,j}( \bm{u} ) \bigr|
  \to 0
  \qquad (n \to \infty).
\]
\end{compactenum}
\end{condition}

The partial derivatives $\dot C_{\infty,j}$ are defined as $0$ for $u_j\in\{0,1\}$. For $\bm{u} \in [0, 1]^d$ and $j \in \{1, \ldots, d\}$, write $\bm{u}^{(j)} = (1, \ldots, 1, u_j, 1, \ldots, 1)$, with $u_j$ appearing at the $j$th coordinate.

\begin{theorem}[Block maxima empirical copula process]
\label{thm:empcopproc}
Let $(\bm{X}_t)_{t\in \ZZ}$ be a stationary multivariate time series with continuous univariate margins. Assume Conditions~\ref{cond:Cinfty}, \ref{cond:alpha-beta} and Condition~\ref{cond:SC}. If either Condition~\ref{cond:C12}(a) or (b) is satisfied, then 
\[ 
  \Cb_{n,m} = \Cb_{n,m}^{\mathrm{alt}} + o_p(1) \weak \Cb 
\]
in $\ell^\infty([0,1]^d)$, where, for $\bm u \in [0,1]^d$, 
\[
  \Cb(\bm u) = \Cb^\circ(\bm u) - \sum_{j=1}^d \dot C_{\infty,j} (\bm u) \, \Cb^\circ( \bm{u}^{(j)} ).
\]
\end{theorem}

In Theorem~\ref{thm:empcopproc}, the empirical copula process was defined by centering around $C_m$. Of course, one may also want to center around the limit, $C_\infty$.

\begin{corollary}[Centering by the limit copula]
\label{cor:empcopproc}
Let $(\bm{X}_t)_{t\in \ZZ}$ be a stationary multivariate time series with continuous univariate margins. Assume Conditions~\ref{cond:Cinfty}, \ref{cond:alpha-beta} and Condition~\ref{cond:SC}. If also
\begin{equation}
\label{eq:Gamma}
  \lim_{n \to \infty} \sqrt{k} ( C_m - C_\infty ) = \Gamma \qquad \text{in $\ell^\infty([0, 1]^d)$},
\end{equation}
then, in $\ell^\infty([0, 1]^d)$ and with $\Cb$ as in Theorem~\ref{thm:empcopproc},
\[
  \sqrt{k} ( \hat{C}_{n,m} - C_\infty ) 
  =
  \sqrt{k} ( \hat{C}_{n,m}^{\mathrm{alt}} - C_\infty ) + o_p(1)
  \weak 
  \Cb + \Gamma.
\]
\end{corollary}

Note that the limit $\Gamma$ in~\eqref{eq:Gamma} is continuous, being the uniform limit of a sequence of continuous functions. In Section~\ref{subsec:OPC} below, we work out an example for which Equation~\eqref{eq:Gamma} is satisfied with a non-trivial limit function $\Gamma$.

\subsection{Estimating the Pickands dependence function}
\label{subsec:minidi}

For strongly mixing sequences, the limit copula $C_\infty$ is an extreme value copula \citep[Theorem~4.2]{Hsi89}. Inference on the Pickands dependence function $A_\infty$ in \eqref{eq:Pickands} can then be based on the empirical copula $\hat{C}_{n,m}$ of the block maxima.

Rank-based inference for the Pickands dependence function based on i.i.d.\ samples whose underlying distribution has an extreme value copula has drawn some attention recently \citep{GenSeg09, BucDetVol11, GudSeg12, BerBucDet13, peng:qian:yang:2013}. What the estimators have in common is that they can all be written as weighted integrals with respect to the empirical copula. The asymptotic behavior of all these estimators can then be derived from the weak convergence of the usual empirical copula process. In the following we will exemplarily extend the results on the minimum-distance estimator to the present setting of estimation from block maxima.

For the definition of the estimator, note that,
for any probability density $p$ on $(0,1)$ such that the following integral exists, we have
\[
  A_\infty(\bm t) = \int_0^1 \log \{ C_\infty(y^{\bm{t}}) \} \frac{ p(y)}{\log(y)} \, dy,
\]
where we used the notation $y^{ \bm t}=(y^{1-t_1-\dotsc-t_{d-1}}, y^{t_1},\dotsc,y^{t_{d-1}})$.
The last display suggests to estimate $A_\infty$ by the sample analogue
\begin{align}
\label{eq:minidi}
  \widehat{A}_{n,m} \left( \bm t \right) 
  = 
  \int_0^1
  \log\{ \tilde{C}_{n,m} ( y^{\bm{t}} ) \} \frac{ p(y)}{\log(y)} \, dy,
\end{align}  
where $\tilde C_{n,m} = \max\{ k^{-\gamma}, \hat C_{n,m}(\bm u) \}$ with some $\gamma >1/2$ to be specified later; the latter modification is needed to avoid the logarithm of zero. For the case of i.i.d.\ samples, the estimator in \eqref{eq:minidi} is exactly as defined in \cite{BucDetVol11, BerBucDet13}, where it is motivated as a minimum distance estimator.

\begin{theorem}[Asymptotic normality]
\label{thm:minidi}
Let $(\bm{X}_t)_{t\in \ZZ}$ be a stationary multivariate time series with continuous univariate margins. Suppose that Condition~\ref{cond:Cinfty}, \ref{cond:alpha-beta} and~\ref{cond:SC} are met and that $\sqrt k(C_m - C_\infty) \to \Gamma$, uniformly.
If the weight function $p:(0,1) \to [0,\infty)$ satisfies
\begin{align}
\label{eq:wf}
  \int_0^1 y^{-\lambda}\frac{p(y)}{|\log (y)|} \,dy < \infty \text{  for some } \lambda>1,
\end{align}
then, for any $\gamma \in (\tfrac{1}{2}, \tfrac{\lambda}{2})$, in the space $\ell^\infty(\Delta_{d-1})$ equipped with the supremum distance,
\begin{align*}
  \Ab_{n} = \sqrt{k} ( \widehat{A}_{n,m} - A_{\infty} ) &\weak \Ab_\infty,
\end{align*}
where the limiting process $\Ab_\infty$ on $\Delta_{d-1}$ can be represented as
\[
   \Ab_{\infty} (\bm t )
  = 
  \int_0^1 \frac{\Cb (y^{ \bm t} ) + \Gamma(y^{ \bm t})}
  {C_\infty ( y^{ \bm t} ) } \frac{p (y)}{\log(y)} \, dy .
\]
\end{theorem}

In fact, in Appendix~\ref{subsec:proofs:minidi} and at no additional cost, we will show a more general result that allows for weight functions inside the integral in \eqref{eq:minidi} that may also depend on $\bm t$, see Theorem~\ref{theo:winf}. In the i.i.d.~case in \cite{BerBucDet13}, this result proved useful for the development of a test for extreme value dependence.

A useful class of weight functions is given by $p_\kappa(y) = (\kappa+1)^2 \times y^\kappa  \times |\log(y)|$ for some $\kappa > 0$, see Example~2.5 in \cite{BucDetVol11}. Condition \eqref{eq:wf} is obviously satisfied for any $\kappa>0$.

As it is the case for most of the available estimators for Pickands dependence functions, $\widehat{A}_{n,m}$ is itself not a Pickands dependence function. A unifying approach to enforce the necessary and sufficient shape constraints has been proposed in \cite{FilGuiSeg08} and \cite{GudSeg12}. A simple additive boundary correction will be employed in the simulation Section~\ref{sec:num}, see formula~\eqref{eq:abc}.

\section{Examples}
\label{sec:examples}

This section is devoted to the verification of Conditions~\ref{cond:Cinfty}, \ref{cond:alpha-beta} and \ref{cond:C12} in specific models. Regarding Condition~\ref{cond:SC}, please see the paragraph right after the statement of that condition.

With respect to Condition~\ref{cond:Cinfty} note that, for multivariate Gaussian time series whose cross-correlation function satisfies a certain summability condition, \citet{amram:1985} and \citet{Hsi89} show that the limit $C_\infty$ is the independence copula. For most of the common time series models, however, it is already hard to obtain convenient expressions for the copula $C_1$ of the stationary distribution, let alone for the one of the block maximum distribution, $C_m$, and for the limit $C_\infty$. Sections~\ref{subsec:MM} and \ref{subsec:RR} deal with two particular examples where Conditions~\ref{cond:Cinfty} and \ref{cond:alpha-beta} are satisfied.

Section~\ref{subsec:OPC} investigates Condition~\ref{cond:C12}~(a), and in particular its strengthening in equation~\eqref{eq:Gamma}, in a special i.i.d.\ situation. 

\subsection{Moving maxima}
\label{subsec:MM}

Consider the discrete-time, $d$-variate moving maxima process $(\bm{U}_t)_{t \in \ZZ}$ of order $p \in \NN = \{1, 2, \ldots\}$ given by
\begin{equation}
\label{eq:Utj}
  U_{tj} = \max_{i = 0, \ldots, p} W_{t-i,j}^{1/a_{ij}} \qquad (t \in \ZZ;\ j = 1, \ldots, d).
\end{equation}
Here $(\bm{W}_s)_{s \in \ZZ}$ is an iid sequence in $(0, 1)^d$, the $d$-variate distribution of $\bm{W}_s$ being the copula $D$. Further, the coefficients $a_{ij}$ ($i = 0, \ldots, p$; $j = 1, \ldots, d$) are nonnegative and satisfy the constraints
\begin{equation}
\label{eq:sumiaij}
  \sum_{i=0}^p a_{ij} = 1 \qquad (j = 1, \ldots, d).
\end{equation}
If $a = 0$ and $w \in (0,1)$, then $w^{1/a} = 0$ by convention. As the notation suggests, the random variables $U_{tj}$ are uniformly distributed on $(0, 1)$. A model with arbitrary continuous margins can be considered by defining $X_{tj} = \eta_j(U_{tj})$, where $\eta_1, \ldots, \eta_d$ are strictly increasing functions from $(0, 1)$ into $\reals$.

Since $\sigma(\bm{U}_t : t \le 0)$ and $\sigma(\bm{U}_t : t \ge p+1)$ are independent, Condition~\ref{cond:alpha-beta} (iii) and (iv) are trivially satisfied. 


Let $C_m$ be the copula of the vector of component-wise maxima $\bm{M}_m = (M_{m,1}, \ldots, M_{m,d})$ given by $M_{m,j} = \max(U_{1j}, \ldots, U_{mj})$ for $j \in \{1, \ldots, d\}$.



For $m \in \NN$, consider the copula, $D_m$, of the vector of componentwise maxima of $m$ independent random vectors with common distribution $D$:
\begin{equation*}
  D_m( \bm{u} ) = \bigl( D(u_1^{1/m}, \ldots, u_d^{1/m}) \bigr)^{m}.
\end{equation*}
We say that $D$ is in the copula domain of attraction of the extreme value copula $D_\infty$ if
\begin{equation}
\label{eq:CDA}
  \lim_{m \to \infty} D_m(\bm{u}) = D_\infty(\bm{u})
  \qquad (\bm{u} \in (0, 1]^d).
\end{equation}
The limit, $D_\infty$, of $C_m$ is in general different from the copula extreme value attractor of $C_1$; see \eqref{eq:C1attractorMM}.

\begin{proposition}
\label{prop:MM:Cm:limit}
Consider the moving maximum process in \eqref{eq:Utj}--\eqref{eq:sumiaij}. If \eqref{eq:CDA} holds, then
\begin{equation}
\label{eq:MM:Cmlimit}
  \lim_{m \to \infty} C_m(\bm{u}) = D_\infty(\bm{u}).
\end{equation}
\end{proposition}

The proof of Proposition~\ref{prop:MM:Cm:limit} is given in Section~\ref{subsec:proofs:MM}. By a refinement of the proof of Proposition~\ref{prop:MM:Cm:limit}, it is actually also possible to derive rates of convergence in \eqref{eq:MM:Cmlimit} given a rate of convergence in \eqref{eq:CDA}. For the sake of brevity, we omit the details.

 \subsection{Random repetition}
\label{subsec:RR}

Consider independent and identically distributed $d$-dimensional random vectors $\bm{X}_0, \bm{\xi}_1, \bm{\xi}_2, \ldots$ and, independently of these, iid indicator random variables $I_1, I_2, \ldots$; write $\Pr( I_t = 1 ) = \theta \in (0, 1]$. For $t = 1, 2, \ldots$, define
\[
  \bm{X}_t =
  \begin{cases}
    \bm{\xi}_t & \text{if $I_t = 1$,} \\
    \bm{X}_{t-1} & \text{if $I_t = 0$.}
  \end{cases}
\]
Then $\bm{X}_0, \bm{X}_1, \ldots$ is a stationary sequence. The process is a simplified version of the doubly stochastic model in \citet[Section~3]{smith:weissman:1994}. By stationarity, we can assume without loss of generality that the process is defined for all $t \in \ZZ$.

Because of the random repetition mechanism, the process $(\bm{X}_t)_t$ is $\beta$-mixing and the mixing coefficients $\beta(n)$ are of the order $O((1-\theta)^n)$ as $n \to \infty$; see Lemma~\ref{lem:RR:beta}.

Let $\bm{M}_m = (M_{m,1}, \ldots, M_{m,d})$ with $M_{m,j} = \max(X_{1,j},\ldots, X_{n,j})$. Further, put $F_m( \bm{x} ) = \Pr[ \bm{M}_m \le x ]$ and $F_{m,j}(x_j) = \Pr[ M_{m,j} \le x_j ]$. Assume the margins $F_{m,j}$ are continuous and let $C_m$ be the copula of $F_m$.

\begin{proposition}
\label{prop:RR:Cm}
For $\bm{u} \in (0, 1]^d$ and as $m \to \infty$,
\begin{multline*}
  C_m( \bm{u} ) \\
  = \{ 1 + o(1) \} \, 
  \biggl[ 
    1 - \theta + \theta
      C_1 
      \biggl( 
	1 + \frac{\log(u_1) + o(1)}{\theta (m-1)}, 
	\ldots, 
	1 + \frac{\log(u_d) + o(1)}{\theta (m-1)}
      \biggr) 
  \biggr]^{m-1}.
\end{multline*}
Consequently, if $C_1$ is in the copula domain of attraction of an extreme value copula $C_\infty$, then also $C_m \to C_\infty$ as $m \to \infty$.
\end{proposition}

The proof of Proposition~\ref{prop:RR:Cm} is given in Appendix~\ref{subsec:proofs:RR}.

\subsection{Rate of convergence in the i.i.d.\ case} \label{subsec:OPC}

For $\theta >0$ and $\beta \ge 1$, 
the outer power transform of a Clayton copula is defined as
\begin{equation} \label{eq:OPC}
  C_{\theta, \beta}(u, v) = [1 + \{ (u^{-\theta} - 1)^\beta + (v^{-\theta} - 1)^\beta \}^{1/\beta}]^{-1/\theta}.
\end{equation}
The copula of the pair of componentwise maxima of an i.i.d.\ sample of size $m$ from a continuous distribution with copula $C_{\theta,\beta}$ is equal to
\begin{equation*}
  \bigl\{ C_{\theta,\beta}(u^{1/m}, v^{1/m}) \bigr\}^m
  =
  C_{\theta/m,\beta}(u, v).
\end{equation*}
As $m \to \infty$, this copula converges to the Gumbel--Hougaard copula with shape parameter $\beta \ge 1$,
\begin{equation}
\label{eq:Gumbel}
  C_{0, \beta}(u, v) 
  :=
  \lim_{m \to \infty} C_{\theta/m,\beta}(u, v)
  = 
  \exp [ - \{ (- \log u)^{\beta} + (- \log v)^{\beta} \}^{1/\beta} ],
\end{equation}
see \cite{ChaSeg09}.
The following result shows that the rate of convergence in \eqref{eq:Gumbel} is $O(1/m)$; its proof is being given in Appendix~\ref{subsec:proofs:OPC}.

\begin{proposition} \label{prop:iidrate}
We have 
\[
  \lim_{m \to \infty} m \, \{ C_{\theta/m,\beta}(u, v) - C_{0,\beta}(u, v) \}
  =
  \theta \, \Gamma_\beta(u, v),
\]
where 
\[
  \Gamma_\beta(u, v)
  =
  \frac{1}{2} \exp \{ - (x^\beta + y^\beta)^{1/\beta} \} \,
  \{ (x^\beta + y^\beta)^{2/\beta} - (x^\beta + y^\beta)^{1/\beta-1} (x^{\beta+1} + y^{\beta+1}) \}
\]
with $x = - \log u$ and $y = - \log v$. The convergence is uniform in $(u, v) \in [0, 1]^2$.
\end{proposition}

As a consequence, if $m \gg n^{1/3}$, then $\sqrt{k} = o(m)$ and Equation~\ref{eq:Gamma} is satisfied with $\Gamma\equiv 0$. If $m \sim cn^{1/3}$ for some positive constant $c$, then $\sqrt k / m\to c^{-3/2}$ and hence Equation~\ref{eq:Gamma} is satisfied with $\Gamma(u,v)= c^{-3/2} \theta \Gamma_\beta(u,v)$. If $m = o(n^{1/3})$, then the block sizes are too small and Condition~\ref{cond:C12}~(a) and Equation \eqref{eq:Gamma} fail. 

By similar arguments as used in the proof of Proposition~\ref{prop:iidrate}, it can be shown that the partial derivatives of $C_{\theta, \beta}$ converge to those of $C_{0, \beta}$, uniformly on the relevant subsets in Condition~\ref{cond:C12}~(b).


\section{Numerical results}
\label{sec:num}

In this section, we investigate the finite-sample performance of the minimum-distance estimator for the Pickands dependence function $A_\infty$ by means of a small simulation study. 

\paragraph{The setup.} 
As a time series model, we consider 
the bivariate moving maximum process $(U_{t,1}, U_{t,2})_{t \in \ZZ}$ of order $1$ as introduced in Section~\ref{subsec:MM}, i.e.,
\begin{align}
\label{eq:setup}
  U_{t,1} = \max(W_{t,1}^{1/a}, W_{t-1,1}^{1/(1-a)}), \qquad
  U_{t,2} = \max(W_{t,2}^{1/a}, W_{t-1,2}^{1/(1-b)}),
\end{align}
where $(a, b) \in (0, 1)^2$ and $(W_{t,1}, W_{t,2})_{t \in \ZZ}$ is a bivariate iid sequence whose marginal distributions are uniform on $(0,1)$ and whose joint cdf is denoted by $D$. In this section, we present results for two different choices for~$D$:  
\begin{enumerate}[1.]
\item 
$D=C_{\theta, \beta}$, the outer power transform of a Clayton copula with parameters $\theta > 0$ and $\beta \ge 1$ as defined in \eqref{eq:OPC}. From the results presented in Section~\ref{subsec:OPC}, independently of $\theta>0$, the max-attractor copula $D_\infty$ is the Gumbel--Hougaard copula, whose Pickands dependence function is given by
\[
  A_\infty(t) = \{ t^\beta + (1-t)^\beta \}^{1/\beta}, \quad \beta \ge 1.
\]
In the simulations, we fixed $\theta=1$. 
\item The $t$-copula with $\nu>0$ degrees of freedom and correlation parameter $\rho\in(-1,1)$, given by 
\begin{multline*}
  D(u,v) = \\ \int_{-\infty}^{ \mathbf{t}_\nu^{-1}(u)} \int_{-\infty}^{\mathbf{t}_\nu^{-1}(v)} 
  \frac{1}{\pi \nu |P|^{1/2}} 
  \frac{\Gamma(\tfrac{\nu}{2} +1 )}{\Gamma(\tfrac{\nu}{2}) } 
  \left( 1+ \frac{x'P^{-1}x}{\nu}\right)^{-\nu/2+1}\, dx_2\,dx_1,
\end{multline*}
where $\mathbf{t}_\nu$ denote the cdf of the univariate $t$-distribution with $\nu$ degrees of freedom and where $P$ denotes the $2\times2$ correlation matrix with off-diagonal element $\rho$. The $t$-copula lies in the max-domain of attraction of the $t$-extreme value copula characterized by the Pickands dependence function
\begin{multline*}
  A_\infty(t) 
  = 
  t \times \mathbf{t}_{\nu+1}(z_t) + (1-t) \times \mathbf{t}_{v+1}(z_{1-t}), \\
  \text{where} \quad z_t = (1+\nu)^{1/2} \left[ \{t/(1-t)\}^{1/\nu} - \rho\right] (1-\rho^2)^{-1/2},
\end{multline*}
see, e.g., \cite{DemMcn05}.
Throughout the simulations we fixed $\nu=4$.
\end{enumerate}
The remaining parameter of the two models ($\beta$ and $\rho$, respectively) are chosen in such a way that the coefficient of upper tail dependence of $D$ varies in the set $\{0.25,0.5,0.75\}$. For $a$ and $b$ in \eqref{eq:setup} we consider all possible combinations such that $(a,b) \in\{0.25,0.5,0.75\}^2$.
Regarding the choice of $n,k$ and $m$,  we either fix $n=1,000$ and consider parameters $m\in \{1,2,\dots,30\}$, or we fix $m=30$ (a month, say) and consider block numbers $k\in \{12,24,36, \dots, 240\}$ (corresponding to one up to 20 years).

\paragraph{The estimators.} In addition to the estimator $\widehat{A}_{n,m_n}$ defined in Section~\ref{subsec:minidi}, we will also consider a simple (additive) boundary correction defined as 
\begin{align}
\label{eq:abc}  
  \widehat{A}_{n,m}^{abc}(t) 
  = 
  \widehat{A}_{n,m}(t) - (1-t) \{ \widehat{A}_{n,m}(0)  - 1\} - t \{ \widehat{A}_{n,m}(1) - 1\}.
\end{align}
Due to the fact that the second and the third summand on the right-hand side of this display are deterministic functions of order $o(k^{-1/2})$, the corrected estimator has the same asymptotic distribution as the uncorrected one.

The estimator $\widehat{A}_{n,m}$ depends on a tuning parameter $\gamma$ and a weight function $p$. We follow the proposals in \cite{BucDetVol11} and consider the choices $\gamma=2/3$ (the estimator is quite robust with respect to this or larger choices) and $p=p_\kappa(y) = (\kappa+1)^2 y^\kappa  |\log(y)|$ with $\kappa=0.5$, see Example~2.5 in \cite{BucDetVol11}. The latter choice yields a good compromise between good finite sample behavior and analytical tractability.

\paragraph{The target values.}
Our simulation study aims at investigating the performance of $\widehat{A}_{n,m}$ and $\widehat{A}_{n,m}^{abc}$ as estimators for $A_\infty$. For that purpose, we choose $21$ points $t_j=j/20$ in the unit interval, $j=0, 1, \dots, 20$, and estimate the summed squared bias 
\[
  B^{(sum)}:=\sum_{j=1}^{19} \{ \Eb[  \widehat{A}_{n,m}(j/20) - A_\infty(j/20) ]\}^2, 
\]
the summed variance 
\[
  \Var^{(sum)}:=\sum_{j=1}^{19} \Var\{ \widehat{A}_{n,m}(j/20) \}
\] 
and the summed mean squared error $\MSE^{(sum)}:=B^{(sum)} + \Var^{(sum)}$ by averaging out over $N=1,000$ repetitions (analogously for $\widehat{A}_{n,m}^{abc}$).

\paragraph{Results and discussion.}
The results are reported only partially. 
Figure~\ref{fig:knmse} is concerned with a fixed sample size $n=1,000$. We plot  $B^{(sum)}$, $\Var^{(sum)}$ and $\MSE^{(sum)}$ against the number of blocks $k_n$ (on a logarithmic scale) for the estimator $\widehat{A}_{n,m}^{abc}$ and for both copula models mentioned above with tail dependence coefficients in $\{0.25, 0.5, 0.75\}$ and with fixed $a=0.25$ and $b=0.5$. For the sake of brevity, we do not show any results for $\widehat{A}_{n,m}$ (they are slightly worse than those for $\widehat{A}_{n,m_n}^{abc}$ in most cases) or for different choices of $a$ and $b$ (they do not reveal any additional qualitative insight compared to the case $a=0.25$ and $b=0.5$). From the pictures we see that, as expected, the variance of the estimator is decreasing in $k$, while the bias is increasing. For $k=n=1,000$, which corresponds to $m=1$, i.e., to not forming blocks at all, it can be shown that the estimators are actually consistent for the function 
\begin{align}
\label{eq:astar}
  A_{1}^{\star} (t)
  = 
 \int_0^1 \log \{ C_1(y^{\bm{t}}) \} \frac{ p_{0.5}(y)}{\log(y)} \, dy
 =
 \frac{9}{4} \int_0^1   \sqrt y \  |\log \{ C_1(y^{\bm{t}}) \} | \, dy,
\end{align}
The latter fact may serve as an explanation for the different magnitude of the bias at the right end of the pictures in Figure~\ref{fig:knmse}. 
More precisely, Table~\ref{tab:L2} states the $L_2$-distances between $A_{1}^{\star}$ and $A_\infty$, which exactly resemble the ordering of the value of the summed squared bias at $k_n=1,000$ over the respective pictures in Figure~\ref{fig:knmse}. Regarding the summed $\MSE$, we observe a rather good and robust performance for values of $k_n$ between $150$ and  $250$, corresponding to block lengths between $4$ and $7$.

\begin{table}[h!]
\begin{tabular}{l  c c c}
\toprule
Tail dependence coefficient & 0.25 & 0.5 & 0.75 \\ 
\midrule
Outer power Clayton copula & $4.62 \times 10^{-2}$ & $1.62 \times 10^{-2} $ & $1.20 \times 10^{-2}$ \\
$t_4$-copula & $2.86 \times 10^{-2}$ & $2.26 \times 10^{-2} $ & $0.80 \times 10^{-2}$\\
\bottomrule 
\end{tabular}
\label{tab:L2}
\caption{$L_2$-distances between $A_\infty$ and $A_{1}^*$, the latter being defined in \eqref{eq:astar}.}
\end{table}

Finally, in Figure~\ref{fig:mse-fixm}, we present simulation results on $\MSE^{(sum)}$ in the case of a fixed $m=30$ and with varying $k_n\in\{12,24, \dots, 240\}$, corresponding to monthly blocks over daily data for 1 up to 20 years. The shape of the functions are as expected; in particular we see that  $\MSE^{(sum)}$ approximately halves when the number of years doubles. Moreover, the pictures reveal a better performance for increasing strength of dependence. The latter may be explained by the fact that, in the extreme case of perfect dependence, the empirical copula is a deterministic function converging at rate $k_n^{-1}$ rather than $k_n^{-1/2}$.

\begin{figure}
\begin{center}
\includegraphics[width=0.49\textwidth]{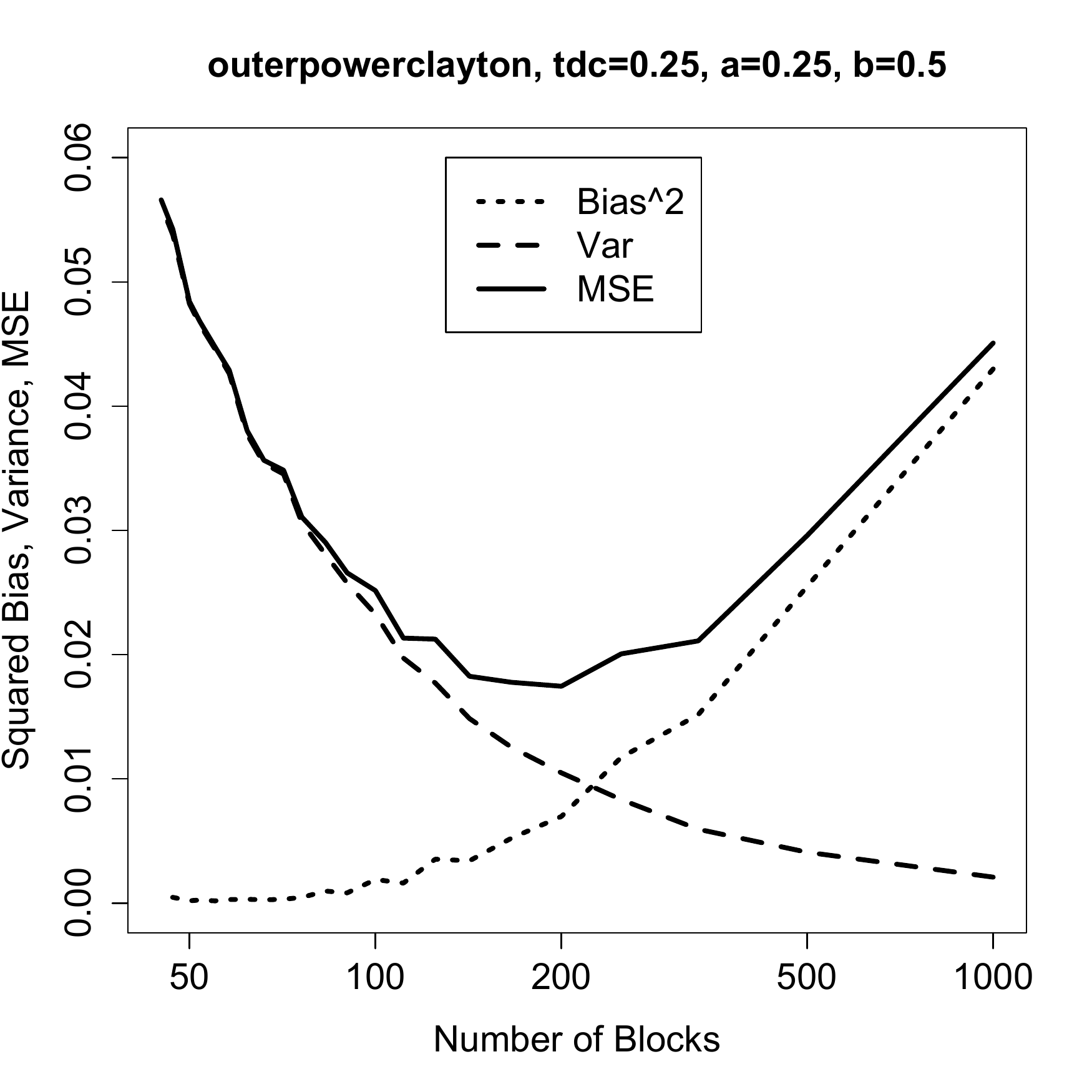}
\includegraphics[width=0.49\textwidth]{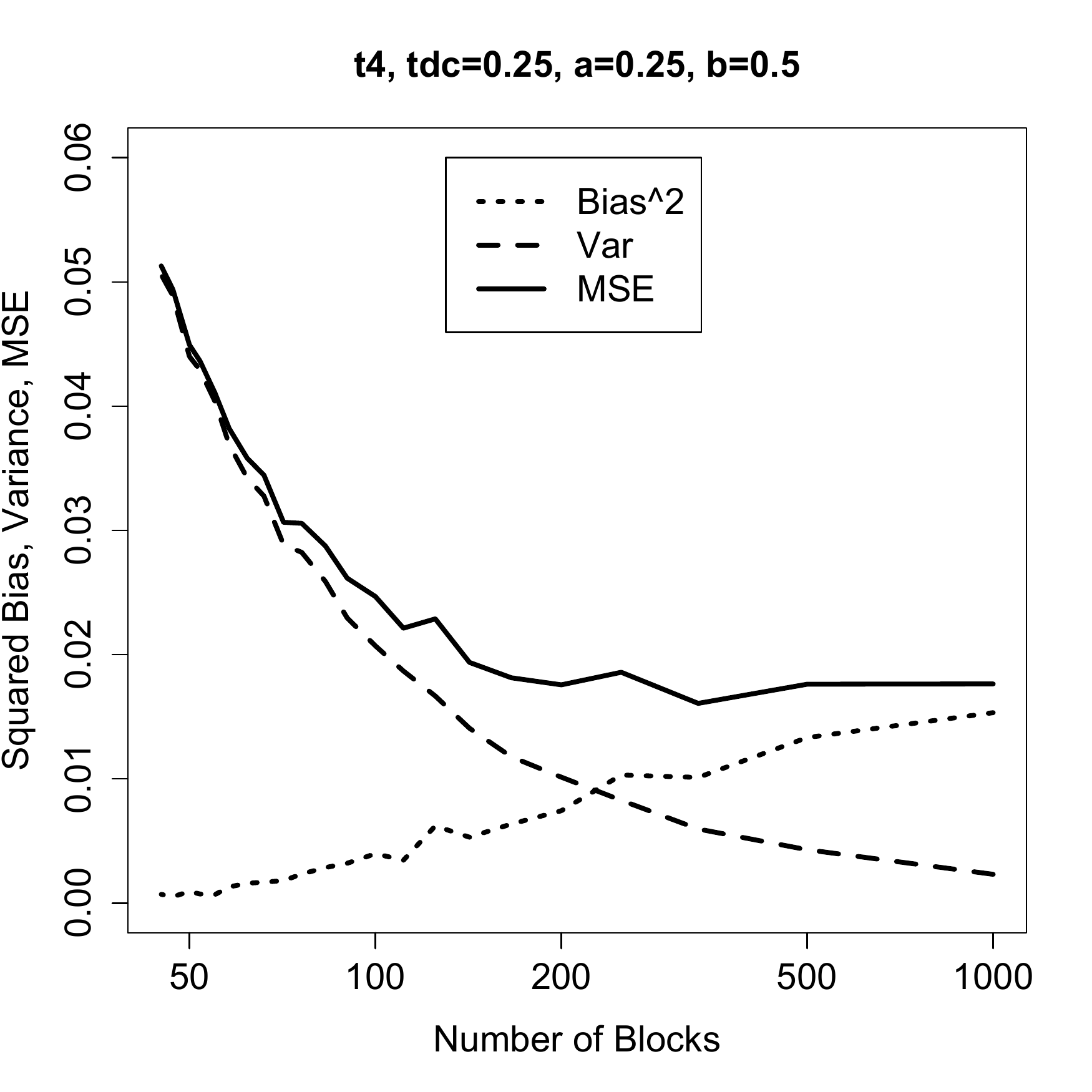}
\includegraphics[width=0.49\textwidth]{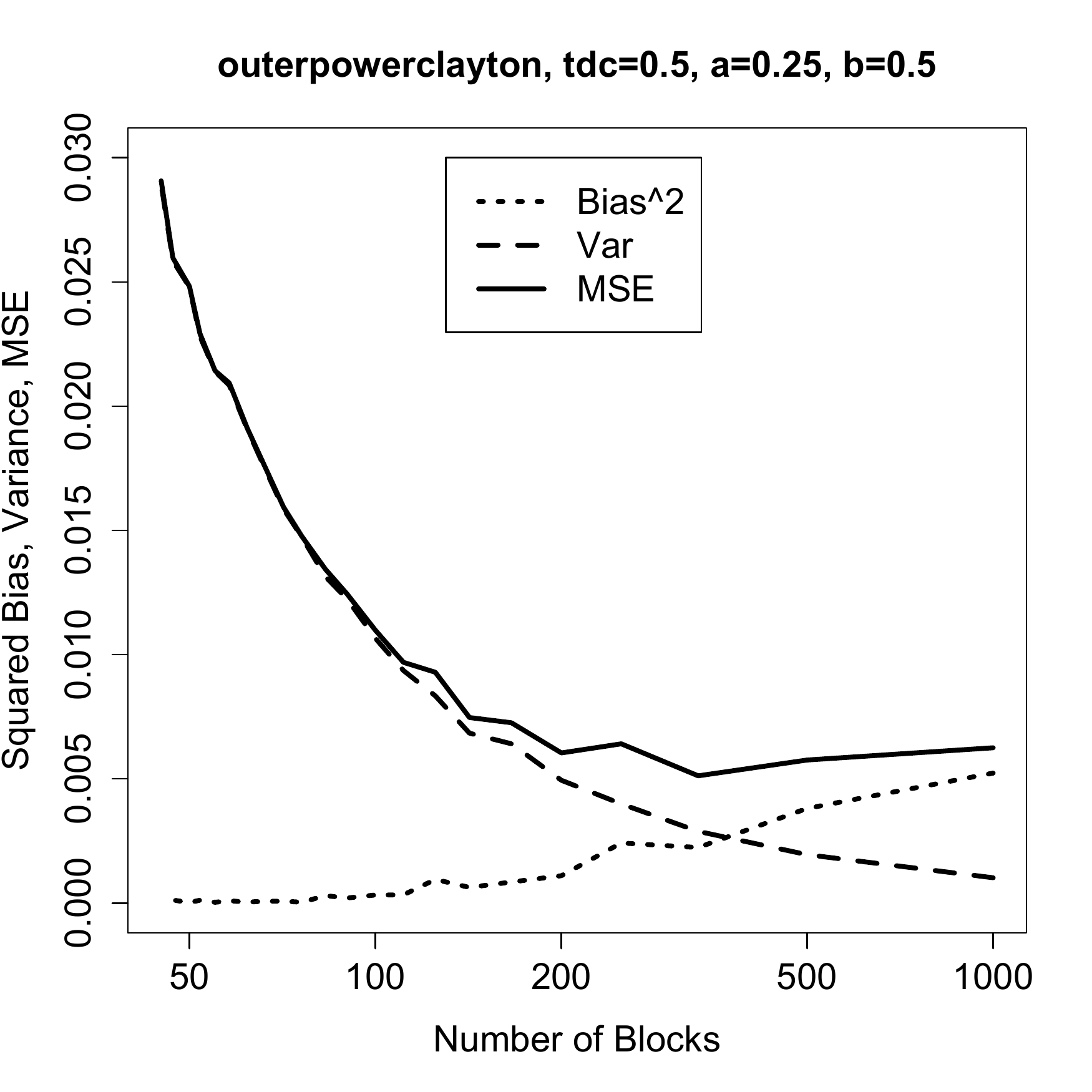}
\includegraphics[width=0.49\textwidth]{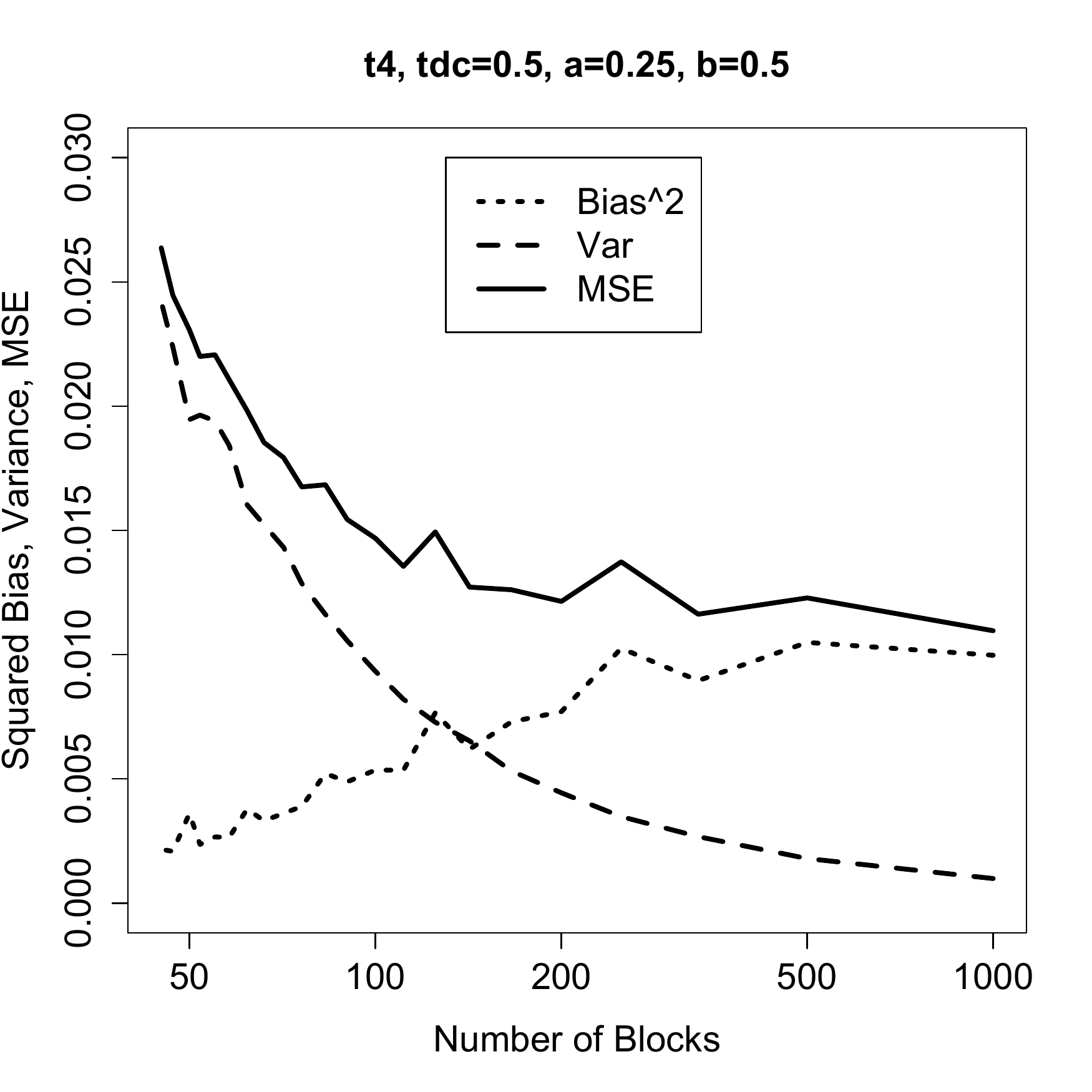}
\includegraphics[width=0.49\textwidth]{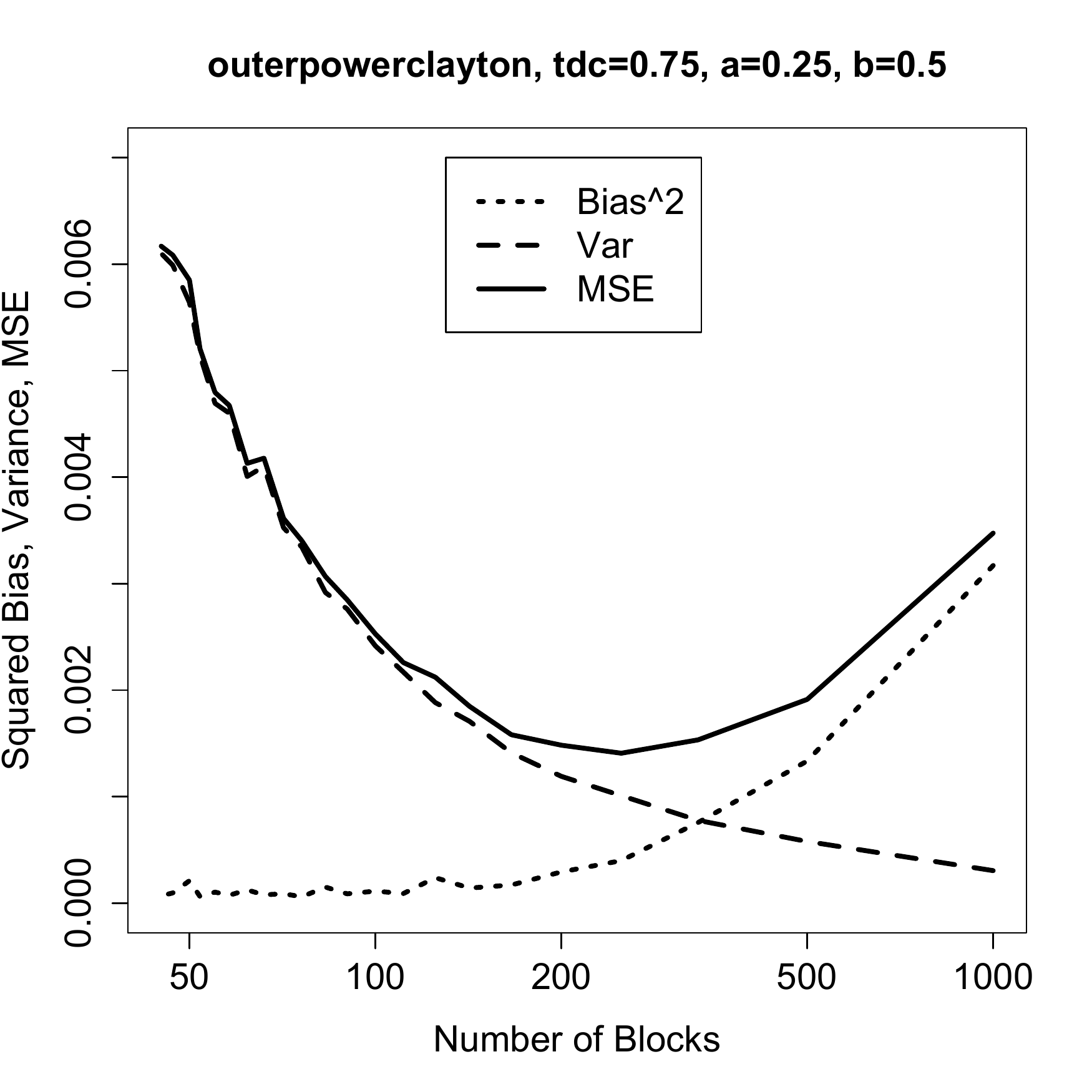}
\includegraphics[width=0.49\textwidth]{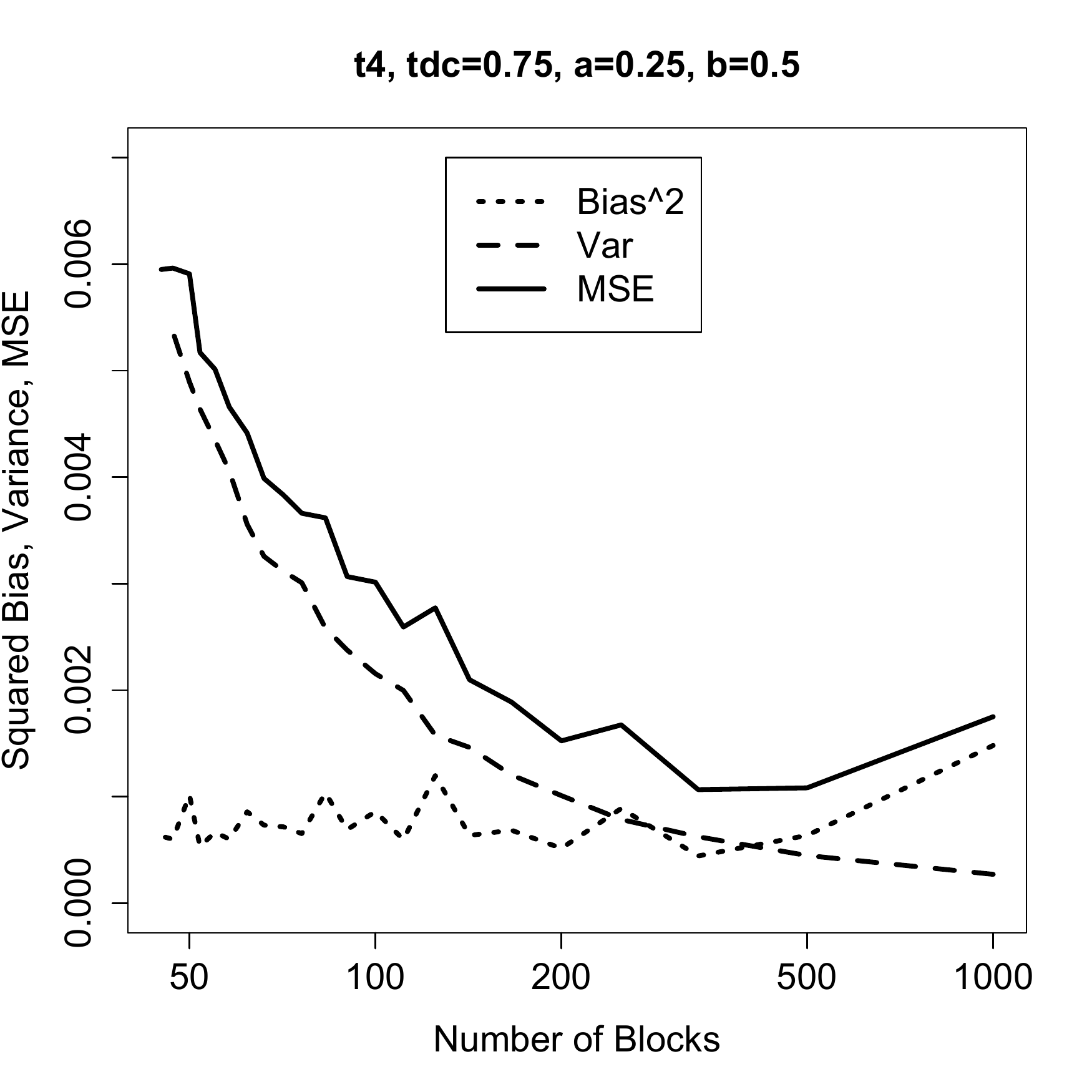}
\end{center}
\vspace{-0.4cm}
\caption{
Simulation results on $B^{(sum)}$, $\Var^{(sum)}$ and $\MSE^{(sum)}$ for fixed $n=1,000$ and varying number of blocks. From top to bottom: tail dependence coefficient $0.25, 0.5$ and $0.75$; left: outer power Clayton copula; right: $t_4$-copula. }
\label{fig:knmse}
\end{figure}

\begin{figure}
\begin{center}
\includegraphics[width=0.49\textwidth]{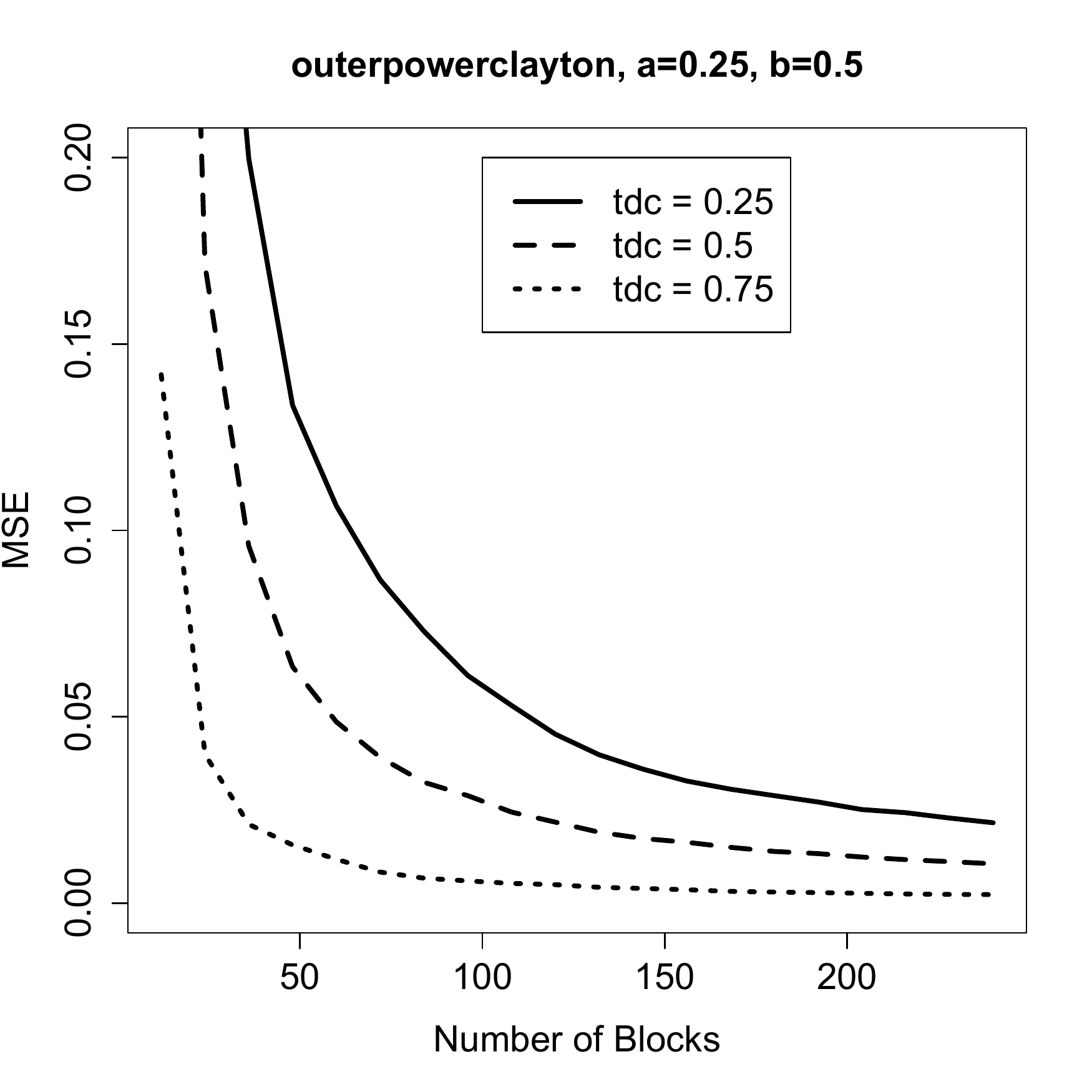}
\includegraphics[width=0.49\textwidth]{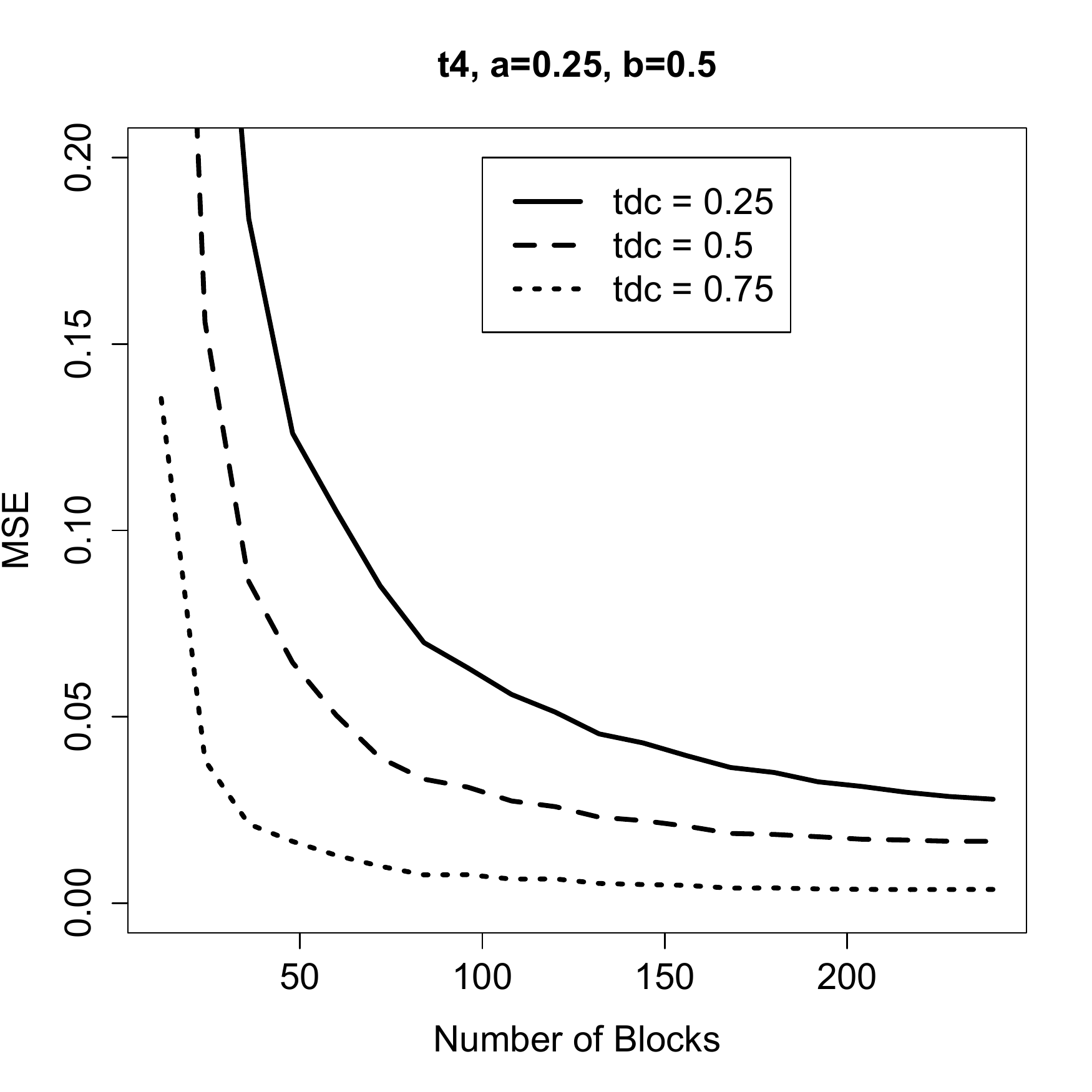}
\end{center}
\vspace{-0.4cm}
\caption{
Simulation results for $\MSE^{(sum)}$, with fixed $m=30$. Left: outer power transform of a Clayton copula, right: $t_4$-copula, both with parameter chosen in such a way that the coefficient of upper tail dependence is given by $0.25$ (solid lines), $0.5$ (dashed lines) or $0.75$ (dotted lines).
}
\label{fig:mse-fixm}
\end{figure}

\section{Conclusion and discussion}
\label{sec:conclusion}

The block maxima method is a time-honoured method in extreme value analysis. In the asymptotic theory, the block maxima are usually modelled as being sampled randomly from an extreme value distribution. In practice, however, the maxima are computed over blocks of finite length. The block length then becomes a tuning parameter, much like the threshold in the peaks-over-threshold method. For large block lengths, the extreme value approximation is accurate, but there are few blocks, leading to large sample variation. Taking smaller blocks augments the number of blocks and thereby reduces the variance of the estimators but at the cost of a potential bias stemming from a bad fit of the extreme value distribution.

The issue is investigated in the context of the nonparametric estimation of the limiting extreme value copula of vectors of componentwise block maxima. The underlying series is supposed to be an absolutely regular, stationary multivariate time series. The sample is partitioned into blocks in such a way that both the block length and the number of blocks tend to infinity. Functional central limit theorems state the asymptotic normality of the empirical copula process and of a rank-based, nonparametric minimum-distance estimator of the Pickands dependence function. The results are illustrated numerically for bivariate moving maximum processes, where the bias-variance trade-off is clearly visible.


The paper leaves ample opportunity for further research into the large-sample theory for the block maxima method for vectors of maxima over blocks of increasing length. We just mention a few possibilities:
\begin{itemize}[--]
\item
The set-up being nonparametric, a convenient way to calculate standard errors would be via bootstrapping the empirical copula process of block maxima. See for instance \cite{bucher:dette:2010} for a review of resampling methods for empirical copula processes.
\item
Often, the extreme value copula is modelled parametrically. In combination with extreme value distributions for the margins, this leads to a parametric model for the block maxima \citep{tawn:1988,tawn:1990}. The asymptotic theory of estimators for the parameters based on a triangular array of block maxima could then be investigated as well.
\item
The minimum distance estimator of the Pickands dependence function is itself not a Pickands dependence function. A way of enforcing the proper shape constraints in arbitrary dimensions is via $L^2$-projection on a parametric sieve \citep{GudSeg12}.
\item
Apart from estimation, there are many interesting hypothesis tests that can be investigated: the goodness-of-fit of a parametric model \citep{genest:etal:2011}, the max-stability hypothesis \citep{kojadinovic:segers:yan:2011}, symmetry or other shape constraints \citep{kojadinovic:yan:2012}, etc.
\item
In Section~\ref{subsec:OPC}, we worked out Condition 3.4 and formula (3.3) in a particular i.i.d.\ situation. We suspect that the convergence rate $O(1/m)$ holds true for more general i.i.d.\ models. 

Within a time series setting, the moving maximum and the random repetition process considered in the paper are a bit artificial. What can one say about the copulas $C_m$ and $C_\infty$ for more common time series models? For GARCH type models, even the computation of the copula, $C_1$, of the stationary distribution is challenging, and the derivation of convergence rates as in Condition~\ref{cond:C12} is even more so.
\end{itemize}

\section*{Acknowledgments}

The authors would like to thank the Associate Editor and the Reviewer for their constructive comments on an earlier version of this manuscript. In particular, we are grateful to the Reviewer for pointing out that a former version of Condition~\ref{cond:alpha-beta} was too restrictive. 

Parts of this paper were written when A.\ B\"ucher was a post-doctoral researcher at Universit\'e catholique de Louvain, Belgium, and at Ruhr-Universit\"at Bochum, Germany.
His reasearch  has been supported in parts by the Collaborative Research Center ``Statistical modeling of nonlinear dynamic processes'' (SFB 823, Project A7) of the German Research Foundation (DFG) and by the IAP research network Grant P7/06 of the Belgian government (Belgian Science Policy), which is gratefully acknowledged.

J. Segers gratefully acknowledges funding by contract ``Projet d'Act\-ions de Re\-cher\-che Concert\'ees'' No.\ 12/17-045 of the ``Communaut\'e fran\c{c}aise de Belgique'' and by IAP research network Grant P7/06 of the Belgian government (Belgian Science Policy).

\appendix

\section{Proofs for Section~\ref{sec:main}}
\label{sec:proofs:main}


\subsection{Proofs for Subsection~\ref{subsec:emp}}
\label{subsec:proofs:emp}

\begin{lemma}
\label{lem:elm}
If $\ell = o(m)$ and $(m/\ell) \, \alpha(\ell) \to 0$, then for every $j \in \{1, \ldots, d\}$ and every $u > 0$,
\[
  \Pr \{ F_{m,j}( M_{\ell,1,j} ) > u \} = O( \ell/m ), \qquad n \to \infty.
\]
\end{lemma}

\begin{proof}
The result is univariate, so we suppress the index $j$ from the notation. Consider the maxima of consecutive subblocks of size $\ell$ contained within the first block of length $m$:
\[
  M_{\ell,1}, \ldots, M_{\ell,\floor{m/\ell}}.
\]
Of these blocks, only keep the ones with an odd index. Since the distribution of $M_{m,1}$ is continuous (all variables $X_t$ having a continuous distribution), we find
\[
  0 < u 
  = \Pr \{ F_m( M_{m,1} ) \le u \}
  \le \Pr \biggl\{ \max_{\substack{1 \le i \le \floor{m/\ell} \\ \text{$i$ is odd}}} F_m( M_{\ell,i} ) \le u \biggr\}.
\]
The odd blocks are separated by a lag $\ell$. Therefore, by induction,
\[
  \biggl|
    \Pr \biggl\{ \max_{\substack{1 \le i \le \floor{m/\ell} \\ \text{$i$ is odd}}} F_m( M_{\ell,i} ) \le u \biggr\}
    -
    \prod_{\substack{1 \le i \le \floor{m/\ell} \\ \text{$i$ is odd}}} \Pr \{ F_m( M_{\ell,i} ) \le u \}
  \biggr|
  \le
  (m/\ell) \, \alpha(\ell).
\]
The number of indices $i$ in the product is at least equal to $\floor{m/\ell}/2$. By stationarity, we obtain
\[
  [1 - \Pr \{ F_m( M_{\ell,1} ) > u \}]^{\floor{m/\ell}/2} \ge u + o(1), \qquad n \to \infty.
\]
But $m/\ell \to \infty$, and thus
\[
  \limsup_{n \to \infty} \frac{m}{\ell} \Pr \{ F_m( M_{\ell,1} ) > u \} < \infty,
\]
as required.
\end{proof}

\begin{proof}[Proof of Theorem~\ref{thm:empproc}]
We first consider convergence of the finite-dimensional distributions of $\Cb_{n,m}^\circ$.
Let $\ell=\ell_n\in\NN$ be some sequence of integers as in Condition~\ref{cond:alpha-beta}. In each block, we clip off smaller blocks of length $\ell$:
\begin{align*}
  M_{m,i,j}^{\tl} &= \max \{ X_{t,j} : t \in (im-m, im-\ell] \cap \ZZ \} \\
  L_{m,i,j}^\tl &= \max \{ X_{t,j} : t \in (im-\ell, im] \cap \ZZ \}.
\end{align*}
Clearly, 
$M_{m,i,j}=\max\{ M_{m,i,j}^\tl, L_{m,i,j}^\tl \}$. Since $\ell = o(m)$, the pieces clipped off, $L_{m,i,j}^\tl$, can be expected to be small. On the other hand, since $\ell \to \infty$,  the clipped blocks $M_{m,1,j}^\tl,\ldots,M_{m,k,j}^\tl$ should be approximately independent. Set
\begin{align*}
  U_{m,i,j}^\tl &= F_{m,j}( M_{m,i,j}^\tl ), &
  \bm{U}_{m,i}^\tl &= (U_{m,i,1}^\tl, \ldots, U_{m,i,d}^\tl),
\end{align*}
and define 
\[
  \hat C_{n,m}^{\circ, \tl} ( \bm u) 
  =
  \frac{1}{k} \sum_{i=1}^k \1 ( \bm U_{m,i}^\tl \le \bm{u} ), \qquad 
  C_m^{\tl}(\bm u) = \Pr(\bm U_{m,1}^\tl \le \bm u)
\]
and  
$
\Cb_{n,m}^{\circ, \tl}(\bm u) =  \sqrt k \{ \hat C_{n,m}^{\circ, \tl} ( \bm u)  - C_{m}^{\tl}(\bm u) \}.
$
Note that $C_m^{\tl}$ is not a copula in general.
First, we are going to show that, for any $\bm u\in[0,1]^d$,
\begin{align} 
\label{eq:clip}
  | \Cb_{n,m}^{\circ}(\bm u) - \Cb_{n,m}^{\circ, \tl}(\bm u) | = o_P(1), \qquad n \to \infty.  
\end{align}
We have
\begin{multline}
\label{eq:CblCb}
  (\Cb_{n,m}^{\circ, \tl} - \Cb_{n,m}^\circ)(\bm u) \\
  = \frac{1}{\sqrt{k}} \sum_{i=1}^k
  [
    \{ \1 ( \bm{U}_{m,i}^\tl \le \bm{u} ) - \1 ( \bm U_{m,i} \le \bm{u} ) \}
    -
    \{ C_m^\tl( \bm{u} ) - C_m( \bm{u} ) \}
  ].
\end{multline}
By the definitions of $C_m$ and $C_{m}^\tl$, the previous expression is centered. Hence it suffices to show that its variance converges to zero. Write
\[
  \Delta_{m,i}^\tl
  = \1 ( \bm{U}_{m,i}^\tl \le \bm{u} ) - \1 ( \bm U_{m,i} \le \bm{u} ).
\]
Since $\bm{U}_{m,i}^\tl \le \bm{U}_{m,i}$ componentwise, we have
\[
  \Delta_{m,i}^\tl
  = \1 ( \bm{U}_{m,i}^\tl \le \bm{u} \not\le \bm{U}_{m,i} ).
\]
By stationarity,
\begin{align*}
  \Var[ (\Cb_{n,m}^\circ - \Cb_{n,m}^{\circ, \tl})(\bm u) ]
  &= \Var\Big[ - \frac{1}{\sqrt{k}} \sum_{i=1}^k \Delta_{m,i}^\tl \Big] \\
  &= \frac{1}{k} \sum_{i_1=1}^k \sum_{i_2=1}^k 
  \Cov[ \Delta_{m,i_1}^\tl, \, \Delta_{m,i_2}^\tl ] \\
  &= \Var[ \Delta_{m,1}^\tl ]
  + 2 \sum_{h = 1}^{k-1} (1 - h/k) \, \Cov[ \Delta_{m,1}^\tl, \, \Delta_{m,h+1}^\tl ]
\end{align*}
Since $\Delta_{m,i}^\tl$ is an indicator variable, its variance is bounded by its expectation. We find, taking the term $h = 1$ out of the sum, 
\[
  \Var[ (\Cb_{n,m}^\circ - \Cb_{n,m}^{\circ, \tl})(\bm u) ]
  \le 3 \expec[ \Delta_{m,1}^\tl ] + 2k \, \alpha( m ).
\]
The second term on the right-hand side converges to $0$ by assumption (iii). 
Moreover,
\begin{align*}
     \Delta_{m,i}^\tl
  &= \left| \1( \bm U_{m,i}^\tl \le \bm u) - \1( \bm U_{m,i} \le \bm u) \right|  \\
  & = 
  \left| \prod\nolimits_{j=1}^d \1(U_{m,i,j}^\tl \le u_j) - \prod\nolimits_{j=1}^d \1( U_{m,i,j} \le u_j) \right| \\
  &  \le
  \sum_{j=1}^d \left| \1(U_{m,i,j}^\tl \le u_j) - \1( U_{m,i,j} \le u_j) \right| 
  = 
  \sum_{j=1}^d \1(U_{m,i,j}^\tl \le u_j < U_{m,i,j}) \\
  & \le
  \sum_{j=1}^d \1\{ F_{m,j}(L_{m,i,j}^\tl ) > u_j \}.
\end{align*}
Therefore, 
$  \expec[\Delta_{m,i}^\tl ]$ can be bounded by
\[
  \sum_{j=1}^d  \Pr\{ F_{m,j}(L_{m,i,j}^\tl)  >  u_j  \} 
  =
  O( \ell / m )
  =
  o(1),
\]
in view of Lemma~\ref{lem:elm} and Condition~\ref{cond:alpha-beta}(ii)--(iii); note that $L_{m,i,j}^\tl$ is equal in distribution to $M_{\ell,1,j}$. The assertion in~\eqref{eq:clip} is proved.

Now, let $\bm u_1, \dots, \bm u_q$ be a finite collection of vectors in $[0,1]^d$. For the fidi-part of the proof, we have to show that
\[
  (\Cb_{n,m}^\circ(\bm u_1) , \dots, \Cb_{n,m}^\circ(\bm u_q) )
  \weak
  (\Cb^\circ(\bm u_1), \dots, \Cb^\circ(\bm u_q) ), 
\]
which, by \eqref{eq:clip}, follows if we prove that
\[
  (\Cb_{n,m}^{\circ,\tl}(\bm u_1) , \dots, \Cb_{n,m}^{\circ, \tl}(\bm u_q) )
  \weak
  (\Cb^\circ(\bm u_1), \dots, \Cb^\circ(\bm u_q) ).
\]
By the Cram\'er--Wold device, the latter is equivalent to 
\[
  Z_n = \sum_{\nu=1}^q c_\nu \Cb_{n,m}^{\circ,\tl}(\bm u_\nu) 
  \weak
  \sum_{\nu=1}^q c_\nu \Cb^{\circ} (\bm u_\nu)  =Z,
\]
for any $c_1 \dots, c_q\in\RR$. Write $Z_n= \sum_{i=1}^k Z_{i,n}$, where
\[
  Z_{i,n} = \frac{1}{\sqrt{k}} 
  \sum_{\nu=1}^q c_\nu 
  \{ \1 ( \bm U_{m,i}^\tl \le \bm u_\nu ) - C_m^\tl(\bm u_\nu) \}.
\]
For $t\in\RR$, let $\psi_{i,n}(t) = \exp(-\mathrm{i}tZ_{i,n})$ with $\mathrm{i}$ the imaginary unit. Note that the characteristic function of $Z_n$ can be written as 
 $t \mapsto \expec \left\{ \prod_{i=1}^{k} \psi_{i,n}(t) \right\}$. Now, for any $t \in \RR$,  we can write
\begin{align*}
  & \left| \expec \left\{ \prod\nolimits_{i=1}^{k} \psi_{i,n}(t) \right\} 
    - \prod\nolimits_{i=1}^{k} \expec \{ \psi_{i,n}(t) \} \right|  \\
  & \hspace{.8cm} \leq \left| \expec\left\{ \prod\nolimits_{i=1}^{k} \psi_{i,n}(t) \right\}  
      - \expec \{ \psi_{1,n}(t) \} \expec \left\{ \prod\nolimits_{i=2}^{k} \psi_{i,n}(t) \right\} \right| \\ 
  & \hspace{.8cm} + \left| \expec \{ \psi_{1,n}(t) \} \right| \left| \expec\left\{ \prod\nolimits_{i=2}^{k_n} \psi_{i,n}(t) \right\} 
    - \expec \{ \psi_{2,n}(t) \} \expec \left\{ \prod\nolimits_{i=3}^{k} \psi_{i,n}(t) \right\} \right| \\
  & \hspace{.8cm} + \ldots  \\ 
  & \hspace{.8cm} + \left| \prod\nolimits_{i=1}^{k-2} \expec \{ \psi_{i,n}(t) \} \right| 
    \left| \expec \left\{ \prod\nolimits_{i=k - 1}^{k}  \psi_{i,n}(t) \right\} 
    - \prod\nolimits_{i=k - 1}^{k} \expec \{ \psi_{i,n}(t) \} \right|.
\end{align*}
Applying $k-1$ times Lemma~3.9 of \cite{DehPhi02}, we obtain 
\begin{multline*}
  \left| \expec\left\{ \prod\nolimits_{i=1}^{k} \psi_{i,n}(t) \right\} - \prod\nolimits_{i=1}^{k} \expec \{ \psi_{i,n}(t) \} \right| \\
  \leq 
  2 \pi k \max_{1 \leq i \leq k} 
      \alpha \left( 
	\sigma \left\{ \psi_{i,n}(t) \right\}, 
	\sigma \left\{ \prod\nolimits_{i'=i+1}^{k}  \psi_{i',n}(t) \right\} 
      \right).
\end{multline*}
Since the maxima $M_{m,i,j}^\tl$ over different blocks $i\ne i'$ are based on observations that are at least $\ell$ observations apart, the right-hand side of the last display is of the order $k \, \alpha(\ell)$. The latter converges to zero by Condition~\ref{cond:alpha-beta}. Hence, we have shown that the finite-dimensional distributions of $\Cb_{n,m}^{\circ, \tl}$ 
show the same asymptotic behavior (with respect to weak convergence)
as those of the process 
\[
  \tilde \Cb_{n,m}^{\circ, \tl}(\bm u) 
  =  
  \sqrt k \{ \tilde{C}_{n,m}^{\circ, \tl} ( \bm u)  - C_{m}^\tl(\bm u) \},
\]
with
\[
  \tilde C_{n,m}^{\circ, \tl} ( \bm u) 
  =
  \frac{1}{k} \sum_{i=1}^k \1 ( \tilde{\bm U}_{m,i}^\tl \le \bm{u} )
\]
and where 
\[
  \tilde{\bm{U}}_{m,i}^\tl = (\tilde U_{m,i,1}, \ldots, \tilde U_{m,i,d}), \qquad
  \tilde{U}_{m,i,j}^\tl = F_{m,j}( \tilde{M}_{m,i,j}^\tl )
\] 
and $\tilde{M}_{m,i,j}^\tl = \max \{ \tilde X_{t,j} : t \in (im-m, im-\ell] \cap \ZZ \}$ is based on a sequence $\tilde{\bm X}_{1}, \dots, \tilde{\bm X}_n$ such that 
\begin{multline*}
  (\tilde{\bm X}_1, \dots, \tilde{\bm X}_m), 
  (\tilde{\bm X}_{m+1}, \dots, \tilde{\bm X}_{2m}), 
  \dots, \\
  (\tilde{\bm X}_{(k-1)m+1}, \dots, \tilde{\bm X}_{km}), 
  (\tilde{\bm X}_{km+1}, \dots, \tilde{\bm X}_{n})
\end{multline*}
are independent and such that each of the $(k+1)$ brackets is equal in law to the same bracket in the original sequence without the tilde. Again applying~\eqref{eq:clip} (which is also valid for the corresponding tilde version based on independent blocks; the proof is even simpler), we obtain that the fidis of $\tilde \Cb_{n,m}^{\circ, \tl}$ converge to the same limit as those of $\tilde \Cb_{n,m}^{\circ}$, which is defined analogously to $\Cb_{n,m}^\circ$, but with ${\bm X}_{1}, \dots, {\bm X}_n$ replaced by $\tilde{\bm X}_{1}, \dots, \tilde{\bm X}_n$.
Assembling everything, the fidis of $\Cb_{n,m}^\circ$ asymptotically behave as those of $\tilde \Cb_{n,m}^\circ$ based on independent blocks. Weak convergence of the latter can easily be deduced from the classical central limit theorem for row-wise independent triangular arrays.

Now, let us prove asymptotic tightness of $\Cb_{n,m}^\circ$. Recall $\beta(n)$ in \eqref{eq:alpha:beta}. By Berbee's coupling Lemma \citep{Ber79, DouMasRio95}, we can construct inductively a sequence 
$
  (\bar{\bm X}_{im+1}, \dots, \bar{\bm X}_{im+m})_{i\ge 0}
$
such that the following three properties hold:
\begin{enumerate}[(i)]
  \item $(\bar{\bm X}_{im+1}, \dots, \bar{\bm X}_{im+m}) \stackrel{d}{=} ({\bm X}_{im+1}, \dots, {\bm X}_{im+m})$ for any $i\ge 0$;
  \item both $(\bar{\bm X}_{2im+1}, \dots, \bar{\bm X}_{2im+m})_{i\ge 0}$ and 
  $(\bar{\bm X}_{(2i+1)m+1}, \dots, \bar{\bm X}_{(2i+1)m+m})_{i\ge0}$ are i.i.d.\ sequences;
  \item $\Pr \{ (\bar{\bm X}_{im+1}, \dots, \bar{\bm X}_{im+m})  \ne ({\bm X}_{im+1}, \dots, {\bm X}_{im+m}) \} \le \beta(m)$.
\end{enumerate}
Let $\bar \Cb_{n,m}^{\circ}$ and $\bar{\bm U}_{m,i}$  be defined analogously to $\Cb_{n,m}^\circ$ and $\bm U_{m,i}$, respectively, but with ${\bm X}_{1}, \dots, {\bm X}_n$ replaced by $\bar{\bm X}_{1}, \dots, \bar{\bm X}_n$. Now, write
\begin{align}
\label{eq:deco-berbe}
  \Cb_{n,m}^{\circ}(\bm u)  
  = 
  \bar \Cb_{n,m}^{\circ}(\bm u)  
  + \{ \Cb_{n,m}^{\circ}(\bm u)  - \bar \Cb_{n,m}^{\circ}(\bm u) \}.
\end{align}
We will show below that the term in brackets on the right-hand side is $o_P(1)$, uniformly in $\bm u\in[0,1]^d$. Then, in order to show asymptotic tightness of $\Cb_{n,m}^{\circ}$, it suffices to show that $\bar \Cb_{n,m}^{\circ}$ is asymptotically tight. 
Write $\bar \Cb_{n,m}^{\circ}=\bar \Cb_{n,m}^{\circ,\mathit{even}}+\bar \Cb_{n,m}^{\circ,\mathit{odd}}$, where
 $\bar \Cb_{n,m}^{\circ,\mathit{even}}$ and $\bar \Cb_{n,m}^{\circ,\mathit{odd}}$ are defined as sums over the even and odd summands of $\bar \Cb_{n,m}^{\circ}$, respectively. Since both of these sums are based on independent summands by property~(ii), they are asymptotically tight by Theorem~2.11.9 in \cite{VanWel96}.

It remains to consider the term in brackets on the right-hand side of \eqref{eq:deco-berbe}. We have
\begin{align*}
  | \bar \Cb_{n,m}^{\circ}(\bm u)  -  \Cb_{n,m}^{\circ}(\bm u) | 
  &\le 
  \frac{1}{\sqrt k} \sum_{i=1}^k \left| \1( \bar{\bm U}_{m,i} \le \bm u ) - \1( \bm U_{m,i} \le \bm u) \right| \\
  &\le 
  \frac{1}{\sqrt k} \sum_{i=1}^k 
    \1 \{ (\bar{\bm X}_{im+1}, \dots, \bar{\bm X}_{im})  \ne ({\bm X}_{im+1}, \dots, {\bm X}_{im}) \}.
\end{align*}
Hence, by Markov's inequality and property~(iii), for any $\eps>0$,
\[
  \Pr\left\{ \sup_{\bm u \in [0,1]^d} | \bar \Cb_{n,m}^{\circ}(\bm u)  -  \Cb_{n,m}^{\circ}(\bm u) | > \eps \right\} 
  \le
  \frac{\sqrt k \beta(m)}{\eps}.
\]
By Condition~\ref{cond:alpha-beta}(iv), we obtain that the second summand on the right-hand side of~\eqref{eq:deco-berbe} is $o_P(1)$ as $n \to \infty$, uniformly in $\bm{u} \in [0, 1]^d$.
\end{proof}

\subsection{Proofs for Subsection~\ref{subsec:empcop}}
\label{subsec:proofs:empcop}

Proposition~\ref{prop:empcopalt} is in fact a corollary to Theorem~\ref{thm:empproc} and Lemma~\ref{lem:alt} below. We prefer to state the lemma independently of the block maxima set-up, as it might be useful in other contexts involving empirical copulas for serially dependent random vectors. To formulate the lemma, we need a bit of notation.

Let $(\bm{Y}_{k,i} = (Y_{k,i,1}, \ldots, Y_{k,i,d}): i = 1, \ldots, k)_{k \in \NN}$ be a triangular array of row-wise stationary, $d$-dimensional random vectors with continuous marginal distribution functions $G_{k,1}, \ldots, G_{k,d}$. Put $\bm{U}_{k,i} = (U_{k,i,j})_{j=1}^d$ with $U_{k,i,j} = G_{k,j}(Y_{k,i,j})$. Let 
\begin{align*}
  \hat{G}_{k,j}(y) &= \frac{1}{k} \sum_{i=1}^k \1(Y_{k,i,j} \le y),
  &
  \hat{G}_k(\bm{y}) &= \frac{1}{k} \sum_{i=1}^k \1(\bm{Y}_{k,i} \le \bm{y}),
\end{align*}
be the marginal and joint empirical distribution functions, respectively, of the sample $\bm{Y}_{k,1}, \ldots, \bm{Y}_{k,k}$. Let $\hat{\bm{U}}_{k,i} = ( \hat{U}_{k,i,j} )_{j=1}^d$ with $\hat{U}_{k,i,j} = \hat{G}_{k,j}(Y_{k,i,j})$. Finally, let $C_k$ be the copula of $\bm{Y}_{k,1}$ and consider the following empirical versions:
\begin{align*}
  \hat{C}_k^\circ( \bm{u} ) 
  &= \frac{1}{k} \sum_{i=1}^k \1 ( \bm{U}_{k,i} \le \bm{u} ), \\
  \hat{C}_k( \bm{u} ) 
  &= \frac{1}{k} \sum_{i=1}^k \1 ( \hat{\bm{U}}_{k,i} \le \bm{u} ), \\
  \hat{C}_k^{\mathrm{alt}}( \bm{u} ) 
  &= \hat{G}_k 
  \bigl( 
    \hat{G}_{k,1}^\leftarrow(u_1), \ldots, \hat{G}_{k,d}^\leftarrow(u_d) 
  \bigr).
\end{align*}

\begin{lemma}
\label{lem:alt}
Consider the set-up in the previous paragraph. If $\sqrt{k} ( \hat{C}_k^\circ - C_k )$ converges weakly in $\ell^\infty([0, 1]^d)$ to a stochastic process with continuous trajectories, then
\[
  \sup_{\bm{u} \in [0, 1]^d} 
  \bigl|
    \hat{C}_k^{\mathrm{alt}}( \bm{u} ) - \hat{C}_k(\bm{u})
  \bigr|
  = o_p( 1/\sqrt{k} ).
\]
\end{lemma}

\begin{proof}
We have
\begin{align*}
  \bigl|
    \hat{C}_k^{\mathrm{alt}}( \bm{u} )
    -
    \hat{C}_k( \bm{u} )
  \bigr|
  &\stackrel{(1)}{\le}
  \sum_{j=1}^d \frac{1}{k} \sum_{i=1}^k
  \bigl|
    \1 \{ Y_{k,i,j} \le \hat{G}_{k,j}^\leftarrow( u_j ) \}
    -
    \1 \{ \hat{G}_{k,j}(Y_{k,i,j}) \le u_j \}
  \bigr| \\
  &\stackrel{(2)}{=}
  \sum_{j=1}^d \frac{1}{k} \sum_{i=1}^k
  \bigl|
    \1 \{ Y_{k,i,j} = \hat{G}_{k,j}^\leftarrow( u_j ) \}
    -
    \1 \{ \hat{G}_{k,j}(Y_{k,i,j}) = u_j \}
  \bigr| \\
  &\stackrel{(3)}{\le}
  \sum_{j=1}^d \frac{1}{k} \sum_{i=1}^k
  \1 \{ Y_{k,i,j} = \hat{G}_{k,j}^\leftarrow( u_j ) \}.
\end{align*}
Explanations:
\begin{compactitem}
\item[(1)]
Write out the definitions of the two versions of the empirical copula and use the inequality $|\prod_j a_j - \prod_j b_j| \le \sum_j |a_j - b_j|$ for numbers $a_j, b_j \in [0, 1]$.
\item[(2)]
Split both indicators into the indicator of a strict inequality and the indicator of an equality. The indicators for the strict inequality are equal, since $x < H^\leftarrow(u)$ if and only if $H(x) < u$ for any distribution function $H$.
\item[(3)]
If $\hat{G}_{k,j}(Y_{k,i,j}) = u_j$, then $Y_{k,i,j} = \hat{G}_{k,j}^\leftarrow( u_j )$. Hence, the second indicator is not larger than the first one.
\end{compactitem}
Fix $j \in \{1, \ldots, d\}$. Let $\hat{C}_{k,j}^\circ$ be the $j$th margin of $\hat{C}_k^\circ$ in \eqref{eq:Ccirc}, that is,
\[
  \hat{C}_{k,j}^\circ( u_j ) 
  = 
  \frac{1}{k} \sum_{i=1}^k \1 ( U_{k,i,j} \le u_j ).
\]
Then we can continue the chain of (in)equalities started in the beginning of the proof by
\begin{align*}
  \bigl|
    \hat{C}_{k}^{\mathrm{alt}}( \bm{u} )
    -
    \hat{C}_{k}( \bm{u} )
  \bigr|
  &\le
  \sum_{j=1}^d 
  \frac{1}{k} \sum_{i=1}^k
  \1 \{ U_{k,i,j} = G_{k,j}(\hat{G}_{k,j}^\leftarrow( u_j )) \} \\
  &\le
  \sum_{j=1}^d 
  \sup_{x \in [0,1]}
  \frac{1}{k} \sum_{i=1}^k
  \1 \{ U_{k,i,j} = x \} \\
  &\le
  \sum_{j=1}^d 
  \sup_{x \in [0,1]}
  \{ \hat{C}_{k,j}^\circ(x) - \hat{C}_{k,j}^\circ(x - 1/k) \} \\
  &\le
  \frac{d}{\sqrt{k}} \, \omega_k( 1/k )
  + \frac{d}{k}
\end{align*}
where $\omega_k$ is the modulus of continuity of $\Cb_k^\circ = \sqrt{k} ( \hat{C}_k^\circ - C_k )$, i.e.
\[
  \omega_k( \delta )
  =
  \sup_{
    \substack{
      ( \bm{x}, \bm{y} ) \in ([0, 1]^d)^2 \\
      \max_j | x_j - y_j | \le \delta
    }
  }
  \bigl|
    \Cb_k^\circ( \bm{x} ) 
    - 
    \Cb_k^\circ( \bm{y} )
  \bigr|,
\]
and where the term $d/k$ comes from the fact that $C_{k,j}$ is the identity on $[0,1]$.
As $\Cb_k^\circ$ converges weakly in $\ell^\infty( [0, 1]^d )$ to a process with continuous trajectories, it follows that $\omega_k( 1/k ) = o_p(1)$.
\end{proof}

\begin{proof}[Proof of Proposition~\ref{prop:empcopalt}]
By Theorem~\ref{thm:empproc}, we can apply Lemma~\ref{lem:alt} to $\bm{Y}_{k,i} = \bm{M}_{m,i}$.
\end{proof}


\begin{proof}[Proof of Theorem~\ref{thm:empcopproc}]
We are going to apply the extended continuous mapping theorem \citep[Theorem 1.11.1]{VanWel96}. Recall the copula mapping, $\Phi$, in \eqref{eq:copulaMapping}, with domain $\Dc_\Phi$. 
Let $\Dc_n$ denote the space of all $\alpha_n \in \ell^\infty([0,1]^d)$ for which $C_m + k^{-1/2} \alpha_n \in \Dc_\Phi$ and define 
\[
  g_n(\alpha_n) = \sqrt k \{ \Phi(C_m + k^{-1/2} \alpha_n) - \Phi(C_m) \}.
\]
Define
\begin{multline} \label{eq:D0}
\Dc_0 = \{ f : [0,1]^d \to \R \mid f \text{ continuous and } f(\bm u) = 0 \text{ for } \bm u=(1, \dots, 1) \\
\text{ or if at least one coordinate of } \bm u \text{ is equal to } 0 \}.
\end{multline}
Since $\Cb^\circ \in \Dc_0$ and since
\[ 
  \sqrt k( \hat{C}_{n,m}^{\mathrm{alt}} - C_m) 
  = g_n \{ \sqrt k ( \hat C_{n,m}^\circ - C_m) \}, 
\]
the assertion will follow from the extended continuous mapping theorem, provided we can show that $g_n(\alpha_n) \to g(\alpha)$ for any $\alpha_n \to \alpha \in \Dc_0$, where $g:\Dc_0 \to \ell^\infty([0,1]^d)$ is defined by
\[
  (g(\alpha))(\bm{u}) 
  = \alpha(\bm{u}) 
  - \sum_{j=1}^d \dot{C}_{\infty,j}(\bm{u}) \, \alpha( \bm{u}^{(j)} ).
\]
Note that $g = \Phi_{C_\infty}'$, the Hadamard derivative of $\Phi$ at $C_\infty$.

Write $\bm{I}_n( \bm{u} ) = ( I_{n1}(u_1), \ldots, I_{nd}(u_d) )$ where
\begin{equation*} 
  I_{nj}(u_j) = ( \id_{[0,1]} + k^{-1/2}\alpha_{nj})^\leftarrow(u_j) 
\end{equation*}
with $\alpha_{nj}(u_j) = \alpha_n(1, \dots, 1, u_j, 1 \dots, 1)$ and with $\id_{[0,1]}$ the identity function on $[0, 1]$. Since 
\[
 \Phi(C_m + k^{-1/2} \alpha_n) = (C_m + k^{-1/2} \alpha_n)(\bm I_n) = C_m (\bm I_n) + k^{-1/2} \alpha_n(\bm I_n)
\] 
we can decompose
\begin{align}
\label{eq:deco_gn}
g_n(\alpha_n) 
&= 
\sqrt k \{ \Phi(C_m + k^{-1/2} \alpha_n) - \Phi(C_m) \}
=
\sqrt k \{ C_m (\bm I_n) + k^{1/2} \alpha_n (\bm I_n) - C_m \} \nonumber \\
&=
\sqrt k \{ C_m (\bm I_n ) - C_m \} + \alpha_n(\bm I_n).
\end{align}
It follows from Vervaat's Lemma, see also formula~(4.2) in \cite{BucVol13}, that 
\begin{align}
\label{eq:inv}
  \sup_{u_j \in [0,1]} 
  \bigl| 
    \sqrt k \{ I_{nj}(u_j) - u_j \} 
    + 
    \alpha( 1, \ldots, 1, u_j, 1, \ldots, 1 )
  \bigr| 
  \to 0.
\end{align}
In particular, by uniform convergence of $\alpha_n$ to $\alpha$ and by uniform continuity of $\alpha$, this implies that the second term on right-hand side of \eqref{eq:deco_gn} converges to $\alpha$, uniformly.

It remains to be shown that the first term on the right-hand side of \eqref{eq:deco_gn} converges to the proper limit, i.e., that
\begin{equation}
\label{eq:firsttermconv}
  \sup_{ \bm{u} \in [0, 1]^d } 
  \biggl| 
    \sqrt k \{ C_m ( \bm{I}_n( \bm{u} ) )  - C_m( \bm{u} ) \}
    + 
    \sum_{j=1}^d \dot{C}_{\infty,j}( \bm{u} ) \, \alpha( \bm{u}^{(j)} ) 
  \biggr|
  = 0.
\end{equation}
The proof of \eqref{eq:firsttermconv} depends on whether we assume Condition~\ref{cond:C12}(a) or (b).

First, we prove \eqref{eq:firsttermconv} under Condition~\ref{cond:C12}(a). Put
\[
  \Delta_n = \sqrt{k} ( C_m - C_\infty ).
\]
We have
\begin{multline*}
  \sqrt k \{ C_m ( \bm{I}_n( \bm{u} ) )  - C_m(\bm{u}) \} \\
  = \sqrt k \{ C_\infty ( \bm{I}_n( \bm{u} ) ) - C_\infty(\bm{u}) \} 
  + \{ \Delta_n ( \bm{I}_n( \bm{u} ) ) - \Delta_n(\bm{u}) \}.
\end{multline*}
It is then sufficient to show that
\begin{align}
\label{eq:firsttermdecomp:a}
  \sup_{ \bm{u} \in [0, 1]^d } 
  \biggl| 
    \sqrt k \{ C_\infty ( \bm{I}_n( \bm{u} ) )  - C_\infty( \bm{u} ) \}
    + 
    \sum_{j=1}^d \dot{C}_{\infty,j}( \bm{u} ) \, \alpha( \bm{u}^{(j)} ) 
  \biggr|
  &\to 0, \\
\label{eq:firsttermdecomp:b}
  \sup_{ \bm{u} \in [0, 1]^d }
  \bigl|
    \Delta_n ( \bm{I}_n( \bm{u} ) ) - \Delta_n(\bm{u})
  \bigr|
  &\to 0.
\end{align}
Convergence in \eqref{eq:firsttermdecomp:a} essentially follows from Condition~\ref{cond:SC}, marginal convergence in \eqref{eq:inv}, the fact that $0 \le \dot{C}_{\infty,j} \le 1$, and $\alpha \in \Dc_0$ with $\Dc_0$ as defined in \eqref{eq:D0}; the proof is in fact the same as the proof for \eqref{eq:firsttermconv} under Condition~\ref{cond:C12}(b) for the special case $m = \infty$. The left-hand side of \eqref{eq:firsttermdecomp:b} is bounded by
$
  \omega ( \delta_n )
$,
where, for $\delta > 0$,
\[
  \omega ( \delta )
  =
  \sup_{n \in \NN} \sup 
  \biggl\{ 
    \bigl| \Delta_n( \bm{u} ) - \Delta_n( \bm{v} ) \bigr| : 
    (\bm{u},\bm{v}) \in ([0, 1]^d)^2, \max_{j=1,\ldots,d} |u_j-v_j| \le \delta 
  \biggr\},
\]
and with
\[
  \delta_n 
  = 
  \max_{j=1,\ldots,d} \sup_{u_j \in [0, 1]} | I_{nj}(u_j) - u_j |.
\]
By \eqref{eq:inv}, $\delta_n \to 0$ as $n \to \infty$. Since the set $\{ \Delta_n : n \in \NN \}$ is relatively compact by Condition~\ref{cond:C12}(a), the functions $\Delta_n$ are uniformly equicontinuous by the Arzel\`a--Ascoli theorem, which means that $\omega ( \delta_n ) \to 0$ as $n \to \infty$.

Next, we prove \eqref{eq:firsttermconv} under Condition~\ref{cond:C12}(b). 
The margins of $C_m$ being uniform on $(0,1)$, the function $C_m$ is Lipschitz; more precisely,
\[
  |C_m(\bm{u}) - C_m(\bm{v})|
  \le \sum_{j=1}^d |u_j - v_j|,
  \qquad (\bm{u}, \bm{v}) \in ([0, 1]^d)^2.
\]
As a consequence, the function
\[
  f : [0, 1] \to [0, 1] : s \mapsto C_m \bigl( \bm{v}_n( \bm{u}, s ) \bigr),
\]
where $\bm{v}_n( \bm{u}, s ) = (1-s) \bm{u} + s \bm{I}_n(\bm{u})$,
is absolutely continuous, a version of its Radon--Nikodym derivative being
\[
  f'(s) = \sum_{j=1}^d \dot{C}_{m,j} \bigl( \bm{v}_n( \bm{u}, s ) \bigr) \, \{ I_{nj}(u_j) - u_j \}.
\]
It follows that
\begin{align*}
  \sqrt k \{ C_m ( \bm{I}_n( \bm{u} ) ) - C_m( \bm{u} )  \}
  &=
  \sqrt{k} \{ f(1) - f(0) \} \\
  &= 
  \sqrt{k} \int_0^1 f'(s) \, \ds \\
  &=
  \sum_{j=1}^d \sqrt{k} \{ I_{nj}(u_j) - u_j \} \int_0^1 \dot{C}_{m,j} \bigl( \bm{v}_n( \bm{u}, s ) \bigr) \, \ds.
\end{align*}
Fix $j = 1, \ldots, d$. We need to show that
\[
  \sup_{ \bm{u} \in [0, 1]^d }
  \biggl|
    \sqrt{k} \{ I_{nj}(u_j) - u_j \} \int_0^1 \dot{C}_{m,j} \bigl( \bm{v}_n( \bm{u}, s ) \bigr) \, \ds
    + \alpha( \bm{u}^{(j)} ) \, \dot{C}_{\infty,j} (\bm{u})
  \biggr|
  \to 0
\]
as $n \to \infty$. 

Given $\eps > 0$, the supremum over those points $\bm{u} \in [0, 1]^d$ such that $u_j \in [0, \delta] \cup [1-\delta, 1]$ can be made smaller than $\eps$ for sufficiently large $n$ by choosing $\delta$ sufficiently small, using the fact that $0 \le \dot{C}_{m,j} \le 1$ and $0 \le \dot{C}_{\infty,j} \le 1$ together with uniform convergence in \eqref{eq:inv} and the assumption that $\alpha \in \Dc_0$. 
In the following, fix such a $\delta$.

Regarding the supremum over $\bm{u} \in [0, 1]^d$ such that $u_j \in [\delta, 1-\delta]$, note that
\begin{align*}
  \int_0^1 \dot{C}_{m,j} \bigl( \bm{v}_n( \bm{u}, s ) \bigr) \, \ds
  &=
  \dot{C}_{\infty,j}( \bm{u} ) \\
  &\mbox{} \quad +
  \int_0^1 \bigl\{ \dot{C}_{\infty,j} \bigl( \bm{v}_n( \bm{u}, s ) \bigr) - \dot{C}_{\infty,j}( \bm{u} ) \bigr\} \, \ds \\
  &\mbox{} \quad +
  \int_0^1 
    \bigl\{ 
      \dot{C}_{m,j} \bigl( \bm{v}_n( \bm{u}, s ) \bigr) 
      - 
      \dot{C}_{\infty,j} \bigl( \bm{v}_n( \bm{u}, s ) \bigr) 
    \bigr\} \, 
  \ds.
\end{align*}
All integrands on the right-hand side converging to zero uniformly over $s \in [0, 1]$ and $\bm{u} \in [0, 1]^d$ such that $u_j \in [\delta, 1-\delta]$, the proof is complete.
\end{proof}

\begin{proof}[Proof of Corollary~\ref{cor:empcopproc}]
As convergence in \eqref{eq:Gamma} implies relative compactness, Condition~\ref{cond:C12}(a) is fulfilled. By Theorem~\ref{thm:empcopproc} and Slutsky's lemma,
\[
  \sqrt{k} ( \hat{C}_{n,m} - C_\infty)
  =
  \Cb_{n,m} + \sqrt{k} ( C_m - C_\infty )
  \weak
  \Cb + \Gamma, \qquad n \to \infty,
\]
as required.
\end{proof}

\subsection{Proofs for Subsection~\ref{subsec:minidi}}
\label{subsec:proofs:minidi}

For the proof of Theorem~\ref{thm:minidi}, we introduce the notation
\begin{align*}
  \Wb_{n, \omega} \left( \bm t \right) 
  =  \sqrt{k}
  \int_0^1 \log \left\{ \frac{\tilde{C}_{n,m} ( \bm y^{ \bm t} ) }
    {C_\infty ( \bm y^{ \bm t} ) } \right\} \omega \left( y,\bm  t  \right) \dy,
\end{align*}
for a measurable weight function $\omega : \left[ 0,1 \right] \times \Delta_{d-1} \to \R $ that may depend on $y$ and $\bm t$. 

\begin{theorem} \label{theo:winf}
Suppose that Conditions~\ref{cond:Cinfty}, \ref{cond:alpha-beta} and~\ref{cond:SC} are met and that $\sqrt k(C_m - C_\infty) \to \Gamma$, uniformly. Assume that there exists a bounded, measurable function $\overline{\omega}: \left[ 0,1 \right] \rightarrow \RR$ such that $\vert \omega \left( y, \bm t \right) \vert \le \overline{\omega} \left( y \right)$ for all $y\in[0,1]$ and all $\bm t \in \Delta_{d-1}$ and such that
\begin{align*}
  \int_0^1  \overline{\omega} \left( y \right) y^{-\lambda} \dy < \infty \text{ for some }  \lambda >1.
\end{align*}
Then, for any 
$\gamma \in \left( \frac{1}{2}, \frac{\lambda}{2} \right)$,
\begin{align*}
  \Wb_{n, \omega} \weak \Wb_{\infty,\omega} 
  ~\text{ in }~ \ell^{\infty} \left( \Delta_{d-1} \right) ,
\end{align*}
as $n \to \infty$,
where  the limiting process is given by
\[
  \Wb_{\infty,\omega} \left( \bm t \right)
  = \int_0^1 \frac{\Cb (\bm y^{ \bm t} ) + \Gamma(\bm y^{ \bm t})}
  {C_\infty ( \bm y^{ \bm t} ) } \, \omega \left( y, \bm t \right) \dy .
\]
\end{theorem}

\begin{proof}[Proof of Theorem~\ref{theo:winf}.]
For $q \in \NN$, let
\begin{align*}
  \Wb_{n, \omega,q} \left( \bm t \right) 
  &=  \sqrt{k}
  \int_{1/q}^1 \log \left\{ \frac{\tilde{C}_{n,m} ( \bm y^{ \bm t} ) }
    {C_\infty ( \bm y^{ \bm t} ) } \right\} \omega \left( y,\bm  t  \right) \dy, \\
  \Wb_{\infty, \omega,q} \left( \bm t \right) 
 & =  
  \int_{1/q}^1 \frac{\Cb (\bm y^{ \bm t} ) + \Gamma (\bm y^{ \bm t}) }
  {C_\infty ( \bm y^{ \bm t} ) } \omega \left( y, \bm t \right) \dy .
\end{align*}
By Lemma B.1 in \cite{BucDetVol11} it suffices to show the following three claims:
\begin{align*}
  \text{(i) } & \Wb_{n, \omega,q}  \weak \Wb_{\infty, \omega,q} \text{ in } \ell^\infty(\Delta_{d-1}) \text{ as } n \to \infty; \\
  \text{(ii) } & \Wb_{\infty, \omega,q}  \weak \Wb_{\infty, \omega} \text{ in } \ell^\infty(\Delta_{d-1}) \text{ as } q \to \infty; \\
  \text{(iii) } & \forall \eps>0: \lim_{q\to\infty } \limsup_{n\to\infty} 
    \Pr\big\{ \sup_{\bm t \in \Delta_{d-1}} | \Wb_{n, \omega,q} (\bm t) - \Wb_{n, \omega} (\bm t) | > \eps\big\} =0.
\end{align*}
The proof of all three assertions follows exactly along the lines of the proof of Theorem~6.1 in \cite{BerBucDet13} and is based on the fact that
\[
  \sqrt{k} ( \hat{C}_{n,m} - C_\infty ) \weak \Cb + \Gamma
\]
in $\ell^\infty([0,1]^d)$ by Corollary \ref{cor:empcopproc}. The details are omitted for the sake of brevity.
\end{proof}

\begin{proof}[Proof of Theorem~\ref{thm:minidi}]
Setting $\omega(y, \bm t) =  p(y)/\log(y)$, the result is a simple corollary of Theorem~\ref{theo:winf}.
\end{proof}

\section{Proofs for Section~\ref{sec:examples}}
\label{sec:proofs:examples}

\subsection{Proofs for Subsection~\ref{subsec:MM}}
\label{subsec:proofs:MM}

The cumulative distribution function, $F_m$, and the copula, $C_m$, of $\bm{M}_m$ are given by
\begin{align*}
  F_m(\bm{u}) = \prod_{s = 1-p}^m D \bigl( (u_j^{\alpha_{mjs}})_{j = 1}^d \bigr), & &
  C_m(\bm{u}) = \prod_{s = 1-p}^m D \bigl( (u_j^{\beta _{mjs}})_{j = 1}^d \bigr), 
\end{align*}
for $\bm{u} \in (0, 1]^d$, where
\begin{align*}
  \alpha_{mjs} &= \max \{ a_{ij} : i = \max(1-s, 0), \ldots, \min(m-s, p) \}, \\
  \alpha_{mj\bullet} &= \sum_{s = 1-p}^m \alpha_{mjs}, \\
  \beta_{mjs} &= \frac{\alpha_{mjs}}{\alpha_{mj\bullet}}.
\end{align*}
The proof is straightforward by direct computation. 

For $m = 1$, we have $\alpha_{1js} = a_{1-s,j}$ and thus 
\[ \alpha_{1j\bullet} = \sum_{s = 1-p}^1 a_{1-s,j} = \sum_{i=0}^p a_{ij} = 1 \]
by \eqref{eq:sumiaij}. We find that 
\begin{equation}
\label{eq:C1}
  C_1(\bm{u}) 
  = F_1(\bm{u})
  = \prod_{i=0}^p D \bigl( (u_j^{a_{ij}})_{j=1}^d \bigr)
\end{equation}
In particular, the random variables $U_{tj}$ are uniformly distributed on $(0, 1)$.

\begin{proof}[Proof of Proposition~\ref{prop:MM:Cm:limit}]
Equation~\eqref{eq:MM:Cmlimit} is a direct consequence of the next sandwich inequality for $C_m$: For $m \in \NN$ such that $m > p$, we have
\begin{equation}
\label{eq:MM:sandwich}
  \bigl( D_{m-p}(\bm{u}) \bigr)^{\frac{m+p}{m-p}} \le C_m(\bm{u}) \le \bigl( D_{m+p}(\bm{u}) \bigr)^{\frac{m-p}{m+p}}.
\end{equation}
We prove \eqref{eq:MM:sandwich}. For $j = 1, \ldots, d$, put
\[
  A_j = \max \{ a_{ij} : i = 0, \ldots, p \}.
\]
Since all $a_{ij}$ are nonnegative and because of \eqref{eq:sumiaij}, we have $0 < A_j \le 1$. Note that
\[
  \alpha_{mjs} = A_j \qquad (s = 1, \ldots, m-p),
\]
whereas for the other $s$, we still have $\alpha_{mjs} \le A_j$. It follows that
\[
  (m-p) A_j \le \alpha_{mj\bullet} \le (m+p) A_j.
\]
In particular, for all $s = 1-p, \ldots, m$,
\[
  \beta_{mjs} = \frac{\alpha_{mjs}}{\alpha_{mj\bullet}} \le \frac{A_j}{(m-p)A_j} = \frac{1}{m-p}.
\]
We find
\begin{align*}
  C_m(\bm{u})
  &\ge \prod_{s = 1-p}^m D(u_1^{1/(m-p)}, \ldots, u_d^{1/(m-p)}) \\
  &= \bigl( D(u_1^{1/(m-p)}, \ldots, u_d^{1/(m-p)}) \bigr)^{m+p} \\
  &= \bigl( D_{m-p}(\bm{u}) \bigr)^{\frac{m+p}{m-p}}.
\end{align*}
On the other hand,
\[
  \beta_{mjs} \ge \frac{A_j}{(m+p)A_j} = \frac{1}{m+p}
  \qquad (s = 1, \ldots, m-p),
\]
from which
\begin{align*}
  C_m(\bm{u})
  \le \prod_{s = 1}^{m-p} D(u_1^{1/(m+p)}, \ldots, u_d^{1/(m+p)}) 
  = \bigl( D_{m+p}(\bm{u}) \bigr)^{\frac{m-p}{m+p}},
\end{align*}
proving also the lower bound. This completes the proof of \eqref{eq:MM:sandwich}.
\end{proof}

The limit $D_\infty$ in Proposition~\ref{prop:MM:Cm:limit} is in general different from the extreme value attractor of $C_1$. Indeed, if \eqref{eq:CDA} holds, then by \eqref{eq:C1},
\begin{multline}
\label{eq:C1attractorMM}
  \bigl( C_1(u_1^{1/m}, \ldots, u_d^{1/m}) \bigr)^{m}
  = \prod_{i=0}^p \bigl\{ D \bigl( (u_j^{a_{ij}/m})_{j=1}^d \bigr) \bigr\}^{m} \\
  \to \prod_{i=0}^p D_\infty \bigl( (u_j^{a_{ij}})_{j=1}^d \bigr), \qquad m \to \infty.
\end{multline}

\subsection{Proofs for Subsection~\ref{subsec:RR}}
\label{subsec:proofs:RR}


Define the $\phi$-mixing coefficient of $(\bm X_t)_t$ as 
\[
  \phi(n) = \sup_{t \in \ZZ} \sup\left\{   | \Pr(A \mid B) - \Pr(A) | : B \in \Fc_{-\infty}^t, A \in \Fc_{t+n}^\infty, \Pr(B)>0\right\} 
\] 
and note that $\beta(n) \le \phi(n)$ \citep{Bra05}. Because of the random repetition mechanism, the process $(\bm{X}_t)_t$ is $\phi$-mixing and the mixing coefficients $\phi(n)$ decay to $0$ geometrically.

\begin{lemma}
\label{lem:RR:beta}
Let $t \in \ZZ$ and $n \in \NN$, and let $B \in \Fc_{-\infty}^t$ with $\Pr(B) > 0$ and $A \in \Fc_{t+n}^\infty$. Then
\[
  | \Pr(A \mid B) - \Pr(A) | \le 2 (1-\theta)^n.
\]
\end{lemma}

\begin{proof}
Consider the event $Q$ that $n$ consecutive repetitions occur at times $t+1, \ldots, t+n$, that is,
\[
  Q = \bigcap\nolimits_{i = 1}^n \{ I_{t+i} = 0 \}.
\]
Note that $Q$ is independent of $A$ and $B$ and that $\Pr(Q) = (1-\theta)^n$. We have
\begin{multline*}
  |\Pr(A \mid B) - \Pr(A)| \\
  \le \Pr(A \cap Q \mid B) + \Pr(A \cap Q) + |\Pr(A \cap Q^c \mid B) - \Pr(A \cap Q^c)| \\
  \le 2 \Pr(Q) = 2 (1-\theta)^n.
\end{multline*}
The inequality follows from the independence of $Q$ and $B$ and the independence of $A \cap Q^c$ and $B$. 
\end{proof}

\begin{proof}[Proof of Proposition~\ref{prop:RR:Cm}]
For an integer $m \ge 2$, partition the event $\{ \bm{M}_m \le \bm{x} \}$ into two pieces, according to whether $I_m$ is equal to $1$ or not:
\begin{align*}
  F_m( \bm{x} )
  &=  \Pr[ \bm{M}_m \le \bm{x} ] \\
  &= \Pr[ \bm{M}_{m-1} \le \bm{x}, \, \bm{X}_m \le \bm{x} ] \\
  &= \Pr[ \bm{M}_{m-1} \le \bm{x} ] \, \Pr[ \bm{\xi}_m \le \bm{x} ] \, \theta
  + \Pr[ \bm{M}_{m-1} \le \bm{x} ] \, (1 - \theta) \\
  &= F_{m-1}( \bm{x} ) \, \{ \theta \, F_1( \bm{x} ) + 1 - \theta \}.
\end{align*}
By induction, we find
\[
  F_m( \bm{x} ) = F_1( \bm{x} ) \, [ 1 - \theta \{ 1 - F_1( \bm{x} ) \} ]^{m-1}.
\]
For the marginal distributions, we find accordingly
\[
  F_{m,j}( x_j ) = F_{1,j}( x_j ) \, [ 1 - \theta \{ 1 - F_{1,j}( x_j ) \} ]^{m-1}
\]
for $j = 1, \ldots, d$. 

For $u_j \in (0, 1]$ and $m \ge 2$, we have
\begin{align*}
  u_j 
  &= F_{m,j} \bigl( \Finv_{m,j}(u_j) \bigr) \\
  &= F_{1,j}( \Finv_{m,j}(u_j) ) \, [ 1 - \theta \{ 1 - F_{1,j}( \Finv_{m,j}(u_j) ) \} ]^{m-1} \\
  &\le [ 1 - \theta \{ 1 - F_{1,j}( \Finv_{m,j}(u_j) ) \} ]^{m-1},
\end{align*}
and thus
\begin{align*}
  F_{1,j} \bigl( \Finv_{m,j}(u_j) \bigr) 
  &\ge 1 - \theta^{-1} (1 - u_j^{1/(m-1)}) \\
  &\ge 1 + \frac{1}{\theta(m-1)} \log u_j 
  \to 1 \qquad (m \to \infty).
\end{align*}
Combining the previous two displays, we find
\[
  u_j = \{ 1 + o(1) \} \, [ 1 - \theta \{ 1 - F_{1,j}( \Finv_{m,j}(u_j) ) \} ]^{m-1}
\]
from which
\begin{align*}
  F_{1,j} \bigl( \Finv_{m,j}(u_j) \bigr) 
  &= 1 - \theta^{-1} \bigl( 1 - [u_j \{ 1 + o(1) \}]^{1/(m-1)} \bigr) \\
  &= 1 + \frac{1}{\theta(m-1)} \{ \log(u_j) + o(1) \}
  \qquad (m \to \infty).
\end{align*}
Writing
\[
  \Fbinv_m( \bm{u} ) = \bigl( \Finv_{m,1}(u_1), \ldots, \Finv_{m,d}(u_d) \bigr),
\]
we have, for $\bm{u} \in (0, 1]^d$,
\[
  F_1 \bigl( \Fbinv_m( \bm{u} ) \bigr)
  \ge 1 - \sum_{j=1}^d \bigl\{ 1 - F_{1,j} \bigl( \Finv_{m,j}(u_j) \bigr) \bigr\}
  \to 1 \qquad (m \to \infty).
\]
The copula, $C_m$, of $F_m$ in $\bm{u} \in (0, 1]^d$ is given by
\begin{align*}
  \lefteqn{
  C_m( \bm{u} ) = F_m \bigl( \Fbinv_m( \bm{u} ) \bigr) 
  } \\
  &= F_1 \bigl( \Fbinv_m( \bm{u} ) \bigr)
  \bigl[ 
    1 - \theta + \theta F_1 \bigl( \Fbinv_m( \bm{u} ) \bigr) 
  \bigr]^{m-1} \\
  &= \{ 1 + o(1) \} \, 
  \biggl[ 
    1 - \theta + \theta
      C_1 
      \biggl( 
	1 + \frac{\log(u_1) + o(1)}{\theta (m-1)}, 
	\ldots, 
	1 + \frac{\log(u_d) + o(1)}{\theta (m-1)}
      \biggr) 
  \biggr]^{m-1}.
\end{align*}
If $C_1$ is in the copula domain of attraction of an extreme value copula $C_\infty$ with stable tail dependence function $L$, then
\[
  \lim_{h \searrow 0}
  h^{-1} \{ 1 - C_1(1 - hx_1, \ldots, 1 - hx_d) \}
  = L(x_1, \ldots, x_d)
\]
locally uniformly in $(x_1, \ldots, x_d) \in [0, \infty)^d$. It follows that
\[
  C_m( \bm{u} )
  \to \exp \{ - L( - \log u_1, \ldots, - \log u_d) \}
  = C_\infty( \bm{u} )
  \qquad (m \to \infty)
\]
as required.
\end{proof}

\subsection{Proofs for Subsection~\ref{subsec:OPC}}
\label{subsec:proofs:OPC}

For the proof of Proposition~\ref{prop:iidrate}, we need two lemmas, the proofs of which are elementary and omitted for the sake of brevity.

\begin{lemma}
\label{lem:exp}
Uniformly in $x \ge 0$,
\[
  \lim_{k \to \infty} k \{ (1 + x/k)^{-k} - e^{-x} \} = e^{-x} \frac{x^2}{2}.
\]
\end{lemma}


\begin{lemma}
\label{lem:log}
Uniformly in $u$ belonging to compact subsets of $(0, 1]$,
\[
  \lim_{k \to \infty}
  k \, \{ k \, (u^{-1/k} - 1) + \log u \}
  =
  \frac{1}{2} (\log u)^2.
\]
\end{lemma}


\begin{proof}[Proof of Proposition~\ref{prop:iidrate}]
It is sufficient to show that
\begin{equation}
\label{eq:Gammabeta}
  \lim_{\theta \downarrow 0}
  \theta^{-1} \{ C_{\theta,\beta}(u, v) - C_{0,\beta}(u, v) \}
  = \Gamma_\beta(u, v).
\end{equation}
Note that
\[
  C_{\theta,\beta}(u, v)
  =
  \biggl[ 
    1 + \theta \, \biggl\{ \biggl(\frac{u^{-\theta}-1}{\theta}\biggr)^\beta + \biggl(\frac{v^{-\theta}-1}{\theta}\biggr)^\beta \biggr\}^{1/\beta}
  \biggr]^{-1/\theta}.
\]
Define an `intermediate' function
\[
  D_{\theta,\beta}(u, v)
  =
  \exp \biggl[ - \biggl\{ \biggl(\frac{u^{-\theta}-1}{\theta}\biggr)^\beta + \biggl(\frac{v^{-\theta}-1}{\theta}\biggr)^\beta \biggr\}^{1/\beta} \biggr]
\]
to be interpreted as $0$ if $\min(u, v) = 0$. Write
\[
  C_{\theta,\beta}(u, v) - C_{0,\beta}(u, v)
  =
  \{ C_{\theta,\beta}(u, v) - D_{\theta,\beta}(u, v) \}
  +
  \{ D_{\theta,\beta}(u, v) - C_{0,\beta}(u, v) \}.
\]
We will treat the two parts on the right-hand side separately.

First, by Lemma~\ref{lem:exp}, as $\theta \downarrow 0$,
\begin{multline}
\label{eq:partA}
  \theta^{-1} \, \{ C_{\theta,\beta}(u, v) - D_{\theta,\beta}(u, v) \} \\
  =
  D_{\theta,\beta}(u, v) \, \frac{1}{2} 
  \biggl\{ 
    \biggl(\frac{u^{-\theta}-1}{\theta}\biggr)^\beta 
    + 
    \biggl(\frac{v^{-\theta}-1}{\theta}\biggr)^\beta 
  \biggr\}^{2/\beta}
  + o(1),
\end{multline}
the $o(1)$ term being uniformly in $(u, v) \in [0, 1]^2$. The right-hand side in \eqref{eq:partA} converges uniformly in $(u, v) \in [0, 1]^2$ to
\[
  C_{0,\beta}(u, v) \, \frac{1}{2} \{ (- \log u)^\beta + (- \log v)^\beta \}^{2/\beta},
\]
to be interpreted as $0$ if $\min(u, v) = 0$. Uniform convergence on compact subsets of $(0, 1]^2$ follows from Lemma~\ref{lem:log}, and uniform convergence on the whole of $[0, 1]^2$ follows from the fact that $(u^{-\theta} - 1)/\theta \ge - \log u$ for $u \in (0,1]$ and $\theta > 0$ and thus 
\[ 
  D_{\theta,\beta}(u, v) < C_{0, \beta}(u,v) \le \min(u, v). 
\]

Second, consider the function
\[
  f_\beta(x, y) = \exp \{ - (x^\beta + y^\beta)^{1/\beta} \}, \qquad (x, y) \in [0, \infty]^2,
\]
taking the value $0$ if $\max(x, y) = \infty$. We have
\begin{multline}
\label{eq:partB}
  \theta^{-1} \, \{ D_{\theta,\beta}(u, v) - C_{0,\beta}(u, v) \} \\
  =
  \theta^{-1} \, \biggl\{ f_\beta \biggl( \frac{u^{-\theta}-1}{\theta}, \frac{v^{-\theta}-1}{\theta} \biggr) -f_\beta(-\log u, -\log v) \biggr\}.
\end{multline}
Further,
\[
  \frac{\partial}{\partial x} f_\beta(x, y) = - f_\beta(x, y) \, (x^\beta + y^\beta)^{1/\beta-1} \, x^{\beta-1},
\]
and similarly for $\partial f_\beta(x, y) / \partial y$. In view of Lemma~\ref{lem:log}, we obtain that the expression in \eqref{eq:partB} converges, as $\theta \to 0$, to
\[
  \frac{\partial}{\partial x} f_\beta(x, -\log v) \bigg|_{x = - \log u}
  \,
  \frac{1}{2} (\log u)^2
  +
  \frac{\partial}{\partial y} f_\beta(-\log u, y) \bigg|_{y = - \log v}
  \,
  \frac{1}{2} (\log v)^2.
\]
This expression can be further simplified to
\[
  - C_{0, \beta}(u, v) \, \frac{1}{2} \{ (- \log u)^\beta + (- \log v)^\beta \}^{1/\beta-1} \, \{ (- \log u)^{\beta+1} + (- \log v)^{\beta + 1} \}
\]
Uniform convergence on $[0, 1]^2$ follows from similar arguments as before. 

Collect terms to conclude.
\end{proof}

\small 
\bibliographystyle{chicago}
\bibliography{biblio}

\end{document}